\documentclass[12pt,a4paper]{article}

\usepackage[utf8]{inputenc}
\usepackage[T1]{fontenc}
\usepackage[english]{babel}
\usepackage{lmodern}
\usepackage{microtype}
\usepackage{amsmath, amssymb, amsfonts, amsthm}
\usepackage{mathtools}
\usepackage{bm}
\usepackage{graphicx}
\usepackage{float}
\usepackage{caption}
\usepackage{subcaption}
\usepackage{booktabs}
\usepackage{array}
\usepackage{geometry}
\usepackage{algorithm}
\usepackage{rotating}
\usepackage{booktabs}
\usepackage{algpseudocode}
\usepackage{xcolor}

\usepackage{hyperref}
\hypersetup{
    colorlinks=true,
    linkcolor=blue,
    citecolor=blue,
    urlcolor=blue
}
\usepackage[numbers]{natbib}
\newtheorem{theorem}{Theorem}

\newtheorem{lemma}{Lemma}
\newtheorem{definition}{Definition}

\title{Relaxed Centralized Circumcentered-Reflection Methods
for Convex Feasibility:
Convergence and Finite Termination}

\begin{document}
\date{}
\maketitle
\author{%
    \begin{minipage}{0.95\textwidth}
        \centering
        \normalsize
        Vincent Guigues\textsuperscript{1}
        \quad
        Roger Behling\textsuperscript{2}
        \quad
        Adriana W. Henarejos\textsuperscript{1}

        \vspace{1em}

        {\small
        \textsuperscript{1}School of Applied Mathematics, FGV,
        Rio de Janeiro, Brazil\par
        \vspace{0.3em}
        \textsuperscript{2}Department of Mathematics, UFSC,
        Blumenau, Santa Catarina, Brazil\par
        }

        \vspace{0.8em}

        {\footnotesize
        \texttt{vincent.guigues@fgv.br}
        \quad
        \texttt{rogerbehling@gmail.com}\par
        \vspace{0.2em}
        \texttt{adriana.henarejos@fgv.edu.br}\par
        }
    \end{minipage}%
}


\abstract{Projection methods are among the most widely used techniques for solving convex feasibility problems due to their simplicity and robustness. However, classical projection-based algorithms and existing circumcentered-reflection methods generally provide only asymptotic convergence. In this paper, we introduce a relaxed version of the centralized circumcentered-reflection method (cCRM) obtained through an extrapolation of the circumcenter step. We first analyze the two-set case and then extend the proposed approach to the successive framework for multiple convex sets, including under- and over-relaxation and a backtracking strategy. The convergence analysis establishes Fejér monotonicity and global convergence under suitable conditions, while local error-bound assumptions yield linear convergence and standard regularity assumptions imply superlinear convergence. Motivated by finitely convergent projection methods based on perturbed feasible sets, we further investigate a perturbed relaxed formulation designed to achieve finite termination. Numerical experiments show good practical performance, comparable to and
competitive with the finitely convergent PACA method. The results also provide
clear numerical evidence of finite termination in the tested instances.}




\maketitle

\section{Introduction}

Projection-type methods are classical tools for convex feasibility problems. Given nonempty closed convex sets $C_1,\ldots,C_m\subseteq\mathbb{R}^n$, the objective is to find a point in $C:=\bigcap_{i=1}^m C_i$. Cyclic projections, simultaneous schemes such as Cimmino's method, and related variants are attractive because they typically require only projections onto the individual sets and, under standard assumptions, converge to a feasible point whenever the intersection is nonempty. Their convergence, however, is generally asymptotic, so the iterates need not enter the feasible region in finitely many steps.

Circumcentered-reflection methods provide a geometric alternative to classical projection schemes. They combine reflections with a circumcenter computation, using information from several points to determine the next iterate. The centralized circumcentered-reflection method (cCRM) \cite{Behling2024Centralization} adds a centralization step that permits Fej\'er-type descent estimates for general closed convex sets, while successive cCRM (ScCRM) \cite{Behling2024Successive} extends the construction to several sets by applying pairwise operators according to suitable control rules.

This paper develops relaxed and perturbed variants of this framework. If $T$ denotes the cCRM operator, we consider the relaxed update $z^{k+1}=z^k+\lambda_k\bigl(T(z^k)-z^k\bigr)$. Values $0<\lambda_k<1$, $\lambda_k=1$, and $\lambda_k>1$ correspond, respectively, to under-relaxation, the original cCRM step, and over-relaxation. The analysis shows that the Fej\'er-type decrease is preserved for a suitable parameter range, including controlled over-relaxation, and extends the resulting convergence theory to successive pairwise steps under almost cyclic and most violated constraint controls. Local error-bound and regularity assumptions are then used to obtain linear and superlinear rates.

To address finite termination, we also introduce a perturbation mechanism inspired by PACA \cite{Behling2024Finite}. For constraints $C_i=\{x:f_i(x)\le0\}$, the method works with sets $C_i^k=\{x:f_i(x)+\eta_k\le0\}$, where $\eta_k>0$, $\eta_k\downarrow0$, and $\sum_k\eta_k=+\infty$, together with separating halfspaces and relaxed circumcentered steps. Under the assumptions developed below, nontermination would force the perturbations to be controlled by a geometrically decreasing distance, contradicting their nonsummability. This provides the mechanism behind the finite-termination result.

The remainder of the paper is organized as follows. Section~\ref{sec_prelim} introduces the notation and basic properties underlying the analysis. Section~\ref{sec_sccrm} develops the relaxed cCRM for two sets and its successive extension, including under- and over-relaxation and the backtracking strategy. Section~\ref{sec_lin_sup_con} establishes linear and superlinear convergence rates for Relaxed ScCRM under suitable error-bound and regularity assumptions. Section~\ref{sec_prsccrm} introduces the perturbed relaxed successive framework and analyzes finite termination. Section~\ref{sec_num_exp} presents numerical experiments comparing the proposed variants with classical projection methods, cCRM-based schemes, and PACA on ellipsoidal, polyhedral, and Euclidean-ball feasibility problems. Finally, Section~\ref{sec_conclusions} summarizes the main conclusions and outlines directions for future work.

\section{Preliminaries}\label{sec_prelim}

Throughout the paper, $\mathbb{R}^n$ is endowed with the Euclidean inner product $\langle\cdot,\cdot\rangle$ and induced norm $\|\cdot\|$. For a nonempty closed convex set $A\subseteq\mathbb{R}^n$, let $P_A(x):=\operatorname*{argmin}_{a\in A}\|x-a\|$ denote the Euclidean projection and $R_A:=2P_A-I$ the associated reflection, where $I$ is the identity operator. We use the following standard projection properties throughout the analysis.

\begin{lemma}[Basic properties of projections onto convex sets]
Let $A\subseteq\mathbb{R}^n$ be nonempty, closed, and convex. Then, for all
$x,y\in\mathbb{R}^n$ and every $s\in A$,

\begin{enumerate}
    \item
    \(
        \|P_A(x)-P_A(y)\|^2
        \le
        \|x-y\|^2
        -
        \|(P_A(x)-x)-(P_A(y)-y)\|^2.
    \)

    \item
    \(
        \|P_A(x)-s\|
        \le
        \|x-s\|.
    \)

    \item
    \(
        \|P_A(x)-x\|^2+\|P_A(x)-s\|^2
        \le
        \|x-s\|^2.
    \)
\end{enumerate}

In particular, projection operators are nonexpansive and firmly
quasinonexpansive. Moreover, if $C_1,\ldots,C_m$ are nonempty closed convex
sets with
\[
    C:=\bigcap_{i=1}^m C_i\neq\varnothing, \text{ then}
\]
\[
    \operatorname{dist}(P_{C_i}(x),C)
    \le
    \operatorname{dist}(x,C),
    \qquad i=1,\ldots,m.
\]
\end{lemma}
These inequalities lead to the Fej\'er monotonicity notion used below.

\begin{definition}[Fej\'er monotonicity]
    A sequence $(w^k)_{k\in\mathbb{N}}\subseteq\mathbb{R}^n$ is Fej\'er monotone with respect to a nonempty set $M\subseteq\mathbb{R}^n$ if \(    \|w^{k+1}-w\|
    \le
    \|w^k-w\|, \forall w\in M,\quad \forall k\in\mathbb{N}.\)

\end{definition}

The following standard consequences will be useful in the subsequent analysis.

\begin{lemma}[Properties of Fej\'er-monotone sequences - Thm. 2.16 in \cite{BauschkeBorwein1996}]
Let $M\subseteq\mathbb{R}^n$ be nonempty, and let $(w^k)$ be Fej\'er monotone with respect to $M$. Then $(w^k)$ is bounded, and for every $s\in M$, the sequence $(\|w^k-s\|)$ converges. Furthermore, if $(w^k)$ has a cluster point $w^\star\in M$, then the entire sequence converges to $w^\star$.
\end{lemma}

When the constraint sets are defined by convex inequalities, we use subgradients to construct separating halfspaces and local Lipschitz continuity in the convergence analysis. For a convex function $f$, the set of subgradients at $x$ is denoted by $\partial f(x)$.
 
\begin{definition}[Subgradient]
    Let $f:\mathbb{R}^n\to\mathbb{R}$ be convex. A vector $g\in\mathbb{R}^n$ is a subgradient of $f$ at $x$ if \(f(y)\ge f(x)+\langle g,y-x\rangle, \forall y\in\mathbb{R}^n.\)

\end{definition}
\begin{definition}[Local Lipschitz continuity]
    A function $f:\mathbb{R}^n\to\mathbb{R}$ is locally Lipschitz continuous on $U\subseteq\mathbb{R}^n$ if, for every $z\in U$, there exist $L>0$ and a neighborhood $V$ of $z$ such that \( |f(x)-f(y)|
    \le
    L\|x-y\|, \forall x,y\in V\cap U.\)

\end{definition}

Finite-valued convex functions on $\mathbb{R}^n$ have locally bounded subdifferentials, so for a finite family the associated subgradients are uniformly bounded on bounded sets. We now introduce the control sequences used to select constraint pairs in the successive schemes.

\begin{definition}[Control sequences]\label{control_sequences}
Let $\ell(k),r(k)\in\{1,\ldots,m\}$.

\par\medskip
\noindent
(i) The control is called \textbf{almost cyclic} if there
exists an integer $Q\ge m$ such that every index
$i\in\{1,\ldots,m\}$ occurs at least once among
\[
    \ell(k+1),\ldots,\ell(k+Q),
\]
for every $k$. The case $Q=m$ corresponds to cyclic control.

\par\medskip
\noindent
(ii) Alternatively, the constraints may be selected
according to their violation at the current iterate.
The \textbf{most violated constraint control in distance form}
is defined by
\[
\begin{aligned}
    \ell(k)
    &\in
    \operatorname*{argmax}_{1\le i\le m}
    \operatorname{dist}(z^k,C_i),\\
    r(k)
    &\in
    \operatorname*{argmax}_{1\le i\le m}
    \operatorname{dist}
    \bigl(P_{C_{\ell(k)}}(z^k),C_i\bigr).
\end{aligned}
\]

\par\medskip
\noindent
(iii) For sets represented as
\[
    C_i=\{x\in\mathbb{R}^n:f_i(x)\le0\},
    \qquad i=1,\ldots,m,
\]
where each $f_i:\mathbb{R}^n\to\mathbb{R}$ is convex,
the corresponding \textbf{function-value control}
is given by
\[
\begin{aligned}
    \ell(k)
    &\in
    \operatorname*{argmax}_{1\le i\le m}
    f_i(z^k),\\
    r(k)
    &\in
    \operatorname*{argmax}_{1\le i\le m}
    f_i\bigl(P_{C_{\ell(k)}}(z^k)\bigr).
\end{aligned}
\]
\end{definition}
Thus, in both most violated constraint rules, the first index is selected at the current iterate, while the second is selected after projection onto the first chosen set. These control rules will be combined with the pairwise circumcentered-reflection operator introduced below.

\subsection{Circumcentered-Reflection Methods}

The circumcentered-reflection method (CRM) was introduced by Behling, Bello-Cruz, and Santos \cite{Behling2018CRM} as a geometric acceleration of Douglas--Rachford for two affine subspaces. Subsequent work established linear convergence and block-processing variants \cite{Behling2018Linear, Behling2020Block}, treated a convex set and an affine subspace and extensions through Pierra's product-space reformulation \cite{Behling2021Convex, Pierra1984Product}, and obtained favorable rate comparisons with alternating projections in the analyzed settings \cite{Arefidamghani2021Rates}. A broader Hilbert-space theory developed circumcenters of finite sets, constructions induced by isometries and nonexpansive operators, convergence properties, and links with best-approximation mappings \cite{Bauschke2018Circumcenters, Bauschke2020Isometries, Bauschke2021Nonexpansive, Bauschke2021LinearIsometries, Bauschke2022Best}.

The framework has also been extended to common fixed-point problems for firmly nonexpansive operators \cite{Arefidamghani2023Fixed}, Bregman-distance constructions \cite{OuyangWang2021Bregman, Ouyang2023Bregman}, nonconvex feasibility and wavelet construction \cite{Dizon2022Nonconvex, Dizon2022Wavelets}, Lyapunov interpretations of methods with spiral trajectories \cite{Lindstrom2022Centering}, circumcentric directions of cones \cite{Behling2023Cones}, and finite convergence for structured configurations involving isometries, hyperplanes, and halfspaces \cite{Ouyang2022Finite}. The circumcentered approximate reflection method (CARM) reduces projection cost by replacing the original sets with simpler outer approximations \cite{Araujo2022CARM}.

For general convex sets, however, the original CRM may not be well defined \cite{Behling2018CRM}, and convergence can fail in some configurations \cite{Aragon2020DR}. The centralized circumcentered-reflection method (cCRM) addresses these issues by combining a centralization step with parallel reflections, yielding global convergence and, under additional assumptions, linear and superlinear rates \cite{Behling2024Centralization}. Its successive extension ScCRM selects pairs of constraints and admits corresponding linear and superlinear analyses for the considered control rules \cite{Behling2024Successive}. Complementarily, PACA uses projections onto perturbed halfspaces and guarantees finite convergence under a Slater condition and positive, nonsummable perturbations decreasing to zero \cite{Behling2024Finite}. These developments motivate the relaxed and perturbed centralized circumcentered schemes studied here.

Let $A,B\subseteq\mathbb{R}^n$ be nonempty closed convex sets with $A\cap B\neq\varnothing$. Define the alternating projection operator by $Z_{A, B}:=P_A \circ P_B$. The centralized point associated with $z$ is
\[
\bar{Z}_{A, B}(z):=\frac{1}{2}\bigl(Z_{A, B}(z)+P_B\left(Z_{A, B}(z)\right)\Bigr)=\frac{1}{2}\bigl(P_A P_B(z)+P_B P_A P_B(z)\Bigr).
\]

With this convention, the operator associated with the ordered pair
$(A,B)$ first projects onto $B$ and then onto $A$. In particular,
the operator associated with the pair $(C_{r(k)},C_{\ell(k)})$
first projects onto $C_{\ell(k)}$, consistently with the control
rules defined above. The following result establishes the
centralization property of $\overline Z_{A,B}(z)$.

\begin{definition}[Centralized point]
    A point $u\in\mathbb{R}^n$ is centralized with respect to $A$ and $B$ if
\[
    \langle R_A(u)-u,R_B(u)-u\rangle\le0.
\]
\end{definition}

By \cite[Lemma~2.2]{Behling2024Centralization}, for every $z\in\mathbb{R}^n$, the point $\overline Z_{A,B}(z)$ is centralized with respect to $A$ and $B$ whenever $A\cap B\neq\varnothing$. The cCRM step applies a parallel circumcenter construction to this centralized point.

For points $a,b,c\in\mathbb{R}^n$, let $\mathcal C(a,b,c)$ denote their circumcenter, when it exists, namely the point in $\operatorname{Aff}\{a,b,c\}$ that is equidistant from $a$, $b$, and $c$. The parallel circumcenter operator is defined by
\[
    \mathcal C_{A,B}(u)
    :=
    \mathcal C\bigl(u,R_A(u),R_B(u)\bigr),
\]
whenever the circumcenter exists. The pairwise centralized circumcentered-reflection operator is then \(T_{A,B}(z)
    :=
    \mathcal C_{A,B}\bigl(\overline Z_{A,B}(z)\bigr)\), or, equivalently,
\[
    T_{A,B}(z)
    =
    \mathcal C\Bigl(
        \overline Z_{A,B}(z),
        R_A\bigl(\overline Z_{A,B}(z)\bigr),
        R_B\bigl(\overline Z_{A,B}(z)\bigr)
    \Bigr).
\]

This centralization property leads to the following descent estimate
for cCRM, which quantifies the decrease in distance to any point in
$A\cap B$ and provides the basis for its convergence analysis.

\begin{lemma}[Lemma~2.6 in \cite{Behling2024Centralization}]
\label{lemmarecccrm}
Let $C_i,C_j\subseteq\mathbb{R}^n$ be nonempty closed convex sets with nonempty intersection, and let $T_{C_i,C_j}$ be their centralized circumcentered-reflection operator. Then, for every $z\in\mathbb{R}^n$ and every $s\in C_i\cap C_j$,
\[
    \|T_{C_i,C_j}(z)-s\|^2
    \le
    \|z-s\|^2
    -
    \frac18\|T_{C_i,C_j}(z)-z\|^2.
\]
\end{lemma}

The displacement term in this estimate is the key ingredient in the relaxed descent inequalities which will appear in what follows.

\section{Convergence Analysis of Relaxed cCRM and Relaxed Successive cCRM}\label{sec_sccrm}

As anticipated, we begin with the relaxed two-set setting, where the basic descent estimate is derived and the admissible range of relaxation parameters is identified. We then extend the analysis to successive pairwise steps for multiple sets and, in addition, consider a backtracking strategy for adapting the relaxation along the iterations.

\subsection{Relaxed cCRM for Two Convex Sets}
Let $X,Y\subseteq\mathbb{R}^n$ be nonempty closed convex sets
with $X\cap Y\neq\varnothing$, and let $\tau>0$ be the stopping
tolerance. Following the convention of Section~\ref{sec_prelim},
let $T=T_{Y,X}$ denote the cCRM operator that first projects
onto $X$ and then onto $Y$. The classical cCRM iteration is
\[
    z^{k+1}=T(z^k).
\]
The over-relaxed method retains the projection, centralization,
and reflection steps of cCRM and modifies only the final update:
\[
    z^{k+1}
    =
    T(z^k)+\epsilon_k\bigl(T(z^k)-z^k\bigr),
    \qquad \epsilon_k>0.
\]
This update extrapolates beyond $T(z^k)$ along the direction
$T(z^k)-z^k$. We present this over-relaxed version for
illustration, the subsequent analysis covers both over-relaxed
and under-relaxed steps within the relaxed ScCRM framework. In what follows we will denote $\lambda_k = 1+ \epsilon_k$.

\begin{algorithm}
\caption{Relaxed centralized circumcentered-reflection method}
\label{alg:relaxed-ccrm}
\begin{algorithmic}[1]

\Require Closed convex sets $X,Y\subset\mathbb{R}^n$;
projections $P_X$ and $P_Y$;
initial point $z^1\in\mathbb{R}^n$;
tolerance $\tau>0$;
maximum number of iterations $K_{\max}$

\For{$k=1,2,\ldots,K_{\max}$}

    \State \textbf{MAP step. } $z_{\mathrm{MAP}} \gets P_Y(P_X(z^k))$

    \State \textbf{Centralization. }$z_C \gets
    \dfrac{1}{2}\left(
        z_{\mathrm{MAP}} + P_X(z_{\mathrm{MAP}})
    \right)$

    \State \textbf{Reflections. } $R_X(z_C) \gets 2P_X(z_C)-z_C$, $R_Y(z_C) \gets 2P_Y(z_C)-z_C$

    \State \textbf{Circumcenter step. }$T(z^k) \gets
    \operatorname{circ}\left(
        z_C,
        R_X(z_C),
        R_Y(z_C)
    \right)$

    \State \textbf{Extrapolation paremeter. }$\epsilon_k \gets \dfrac{1}{k}$ (or another admissible sequence).

    \State \textbf{Relaxed update. }
$z^{k+1} \gets
z^k+(1+\epsilon_k)\bigl(T(z^k)-z^k\bigr)$
    \State \textbf{Stopping criterion. }
    \State $\mathrm{res}_{k+1} \gets
    \max\left\{
        \|P_X(z^{k+1})-z^{k+1}\|,
        \|P_Y(z^{k+1})-z^{k+1}\|
    \right\}$

    \If{$\mathrm{res}_{k+1}\leq\tau$}
        \State \Return $z^{k+1}$
    \EndIf

\EndFor

\end{algorithmic}
\end{algorithm}

Algorithm~\ref{alg:relaxed-ccrm} allows both under-relaxed
and over-relaxed steps, depending on the choice of $\epsilon_k$.
Specifically, $-1<\epsilon_k<0$ yields under-relaxation,
$\epsilon_k=0$ recovers the original cCRM step, and
$\epsilon_k>0$ yields over-relaxation.

\subsubsection{Convergence}

Having introduced the relaxed cCRM step, we now establish
its convergence for relaxation sequences $(\epsilon_k)$
satisfying the assumptions stated below. The analysis
relies on the descent estimate of Lemma~\ref{lemmarecccrm}
and covers both under-relaxed and over-relaxed steps.

Here, $k$ denotes the iteration counter, and $k_0$ is an
index from which the prescribed bounds on $\epsilon_k$ hold. Thus, finitely many initial parameter choices may lie
outside these bounds. The following theorem establishes
convergence under this eventual restriction.

\begin{theorem}[Convergence of Relaxed cCRM]\label{theoremconvrel}
Assume that there exist 
\(\eta_1, \eta_2 >0\) such that
\begin{equation}\label{condetaeps}
    -1+\eta_1 \le \varepsilon_k\le \frac18-\eta_2,
    \qquad \forall k\ge k_0.
\end{equation}
Then the sequence \((z^k)\) generated by the
Relaxed cCRM method 
converges to a point $z^{\star} \in X \cap Y$.
\end{theorem}

\begin{proof}
Using Lemma 2.6 in \cite{Behling2024Centralization},
we have
\begin{equation}\label{ineqcCRM}
    \|T(z)-s\|^2
    \le
    \|z-s\|^2
    -
    \frac18\|z-T(z)\|^2,
\end{equation}
for every \(z\in \mathbb R^n\) and every \(s\in X \cap Y\).
Fix \(s\in X \cap Y\). For every \(k\), define
\begin{equation}
    e_k:=T(z^k)-z^k.
\end{equation}
Applying inequality 
\eqref{ineqcCRM}
with \(z=z^k\), we obtain
$
    \|z^k+e_k-s\|^2
    \le
    \|z^k-s\|^2
    -
    \frac18\|e_k\|^2,
$
which can be written
$$
    \|z^k-s\|^2
    +
    2\langle z^k-s,e_k\rangle
    +
    \|e_k\|^2
    \le
    \|z^k-s\|^2
    -
    \frac18\|e_k\|^2. \text{ Therefore,}
$$

\begin{equation}\label{ineq1p1}
    2\langle z^k-s,e_k\rangle
    \le
    -\frac98\|e_k\|^2.
\end{equation}
Using relation \eqref{ineq1p1} and
the definition of the relaxed cCRM method, we
then have
$$
\begin{aligned}
    \|z^{k+1}-s\|^2
    &=
    \|z^k-s+(1+\varepsilon_k) e_k\|^2        \\
    &=
    \|z^k-s\|^2
    +
    2(1+\varepsilon_k) \langle z^k-s,e_k\rangle
    +
    (1+\varepsilon_k)^2\|e_k\|^2             \\
    & \le
    \|z^k-s\|^2
    -
\frac98(1+\varepsilon_k)\|e_k\|^2
    +(1+\varepsilon_k)^2\|e_k\|^2\\ 
&=
    \|z^k-s\|^2
    -
    \lambda_k\left(\frac18-\varepsilon_k\right)\|e_k\|^2.
\end{aligned}
$$
Condition \eqref{condetaeps}
implies
\[
    (1+\varepsilon_k)\left(\frac18-\varepsilon_k\right)\ge \eta_1 \eta_2.
\]
It follows that 
\begin{equation}\label{zkp1rec}
    \|z^{k+1}-s\|^2
    \le
    \|z^k-s\|^2
    -
   \eta_1 \eta_2\|T(z^k)-z^k\|^2.
\end{equation}
In particular,
\begin{equation}\label{inefejer}
    \|z^{k+1}-s\|\le \|z^k-s\|,\;\;\forall s \in X \cap Y,
\end{equation}
so \((z^k)\) is Fejér monotone with respect to \(X\cap Y\). Therefore \((z^k)\) is bounded.

Next, summing  inequalities  
\eqref{zkp1rec}
from \(k=0\) to \(N\), we obtain
\[
    \eta_1\eta_2\sum_{k=0}^{N}\|T(z^k)-z^k\|^2
    \le
    \|z^0-s\|^2-\|z_{N+1}-s\|^2
    \le
    \|z^0-s\|^2.
\]
Letting \(N\to\infty\), we have
$ \sum_{k=0}^{\infty}\|T(z^k)-z^k\|^2<+\infty,$ which implies
\begin{equation}\label{Tzconvsub}
\lim_{k \rightarrow +\infty} e_k = 
  \lim_{k \rightarrow +\infty}  T(z^k)-z^k =  0.
\end{equation}

We now prove that every accumulation point of \((z^k)\) belongs to \(X \cap Y\). Since \((z^k)\) is bounded, it has at least one accumulation point. Let
$\bar z$ be an accumulation point of sequence $z^k$ and
$z^{k_j}$ a subsequence such that 
\begin{equation}\label{convzbar}
  \lim_{j \rightarrow + \infty}  z^{k_j} =\bar z.
\end{equation}
Using \eqref{Tzconvsub}
and \eqref{convzbar}, we deduce that 
\(
 \lim_{j \rightarrow + \infty}   T(z^{k_j})= \bar z.
\)

Let \(z_C^k\) denote the centralized point associated with \(z^k\), namely
\[
    z_{\mathrm{MAP}}^k=P_Y P_X(z^k),
    \qquad
    z_C^k=
    \frac12\bigl(z_{\mathrm{MAP}}^k+P_X(z_{\mathrm{MAP}}^k)\bigr).
\]
Using Lemma 2.4 in \cite{Behling2024Centralization}, we have
\(
    \|z_C^k-s\|^2
    \le
    \|z^k-s\|^2
    -
    \frac14\|z^k-z_C^k\|^2
\)
and using Lemma 2.5 in \cite{Behling2024Centralization}, we have 
\(
    \|T(z^k)-s\|^2
    \le
    \|z_C^k-s\|^2
    -
    \|T(z^k)-z_C^k\|^2.
\)
Adding these inequalities yields
\begin{equation}\label{ineqT1}
\frac14\|z^k-z_C^k\|^2
    +
    \|T(z^k)-z_C^k\|^2
    \le
    \|z^k-s\|^2-\|T(z^k)-s\|^2.
\end{equation}
We can write
$$
\begin{aligned}
\|T(z^k)-s\|^2 
& =  \|z^k+e_k-s\|^2  =  \|z^k-s\|^2 + \|e_k\|^2 + 2 \langle
z^k-s,e_k
\rangle \text{ and therefore }\\
\end{aligned}
$$

\begin{equation}\label{gtozT}
\Big \lvert \|z^k-s\|^2 - \|T(z^k)-s\|^2 \Big \rvert \leq  \|e_k\|^2+ 2\| z^k-s\| \|e_k\|. 
\end{equation}

Combining \eqref{Tzconvsub} and 
\eqref{gtozT} we have that the right-hand-side of \eqref{ineqT1}
converges to 0 as $k \rightarrow +\infty$,
which implies, using \eqref{ineqT1}, that
\begin{equation}\label{tzkct0}
   \lim_{k \rightarrow +\infty} z^k-z_C^k =  0,
    \qquad
   \lim_{k \rightarrow +\infty}  T(z^k)-z_C^k = 0.
\end{equation}
This implies
\begin{equation}\label{zkjqzb}
  \lim_{j \rightarrow +\infty}  z_C^{k_j} =  \bar z.
\end{equation}

Next, by definition of the
iterations of the centralized circumcenter reflection method, 
\[
    T(z^{k_j})
    =
    \mbox{circ}\bigl(
        z_C^{k_j},
        R_X(z_C^{k_j}),
        R_Y(z_C^{k_j})
    \bigr).
\]
By definition of the circumcenter, we also have
\begin{equation}\label{circproper}
    \|T(z^{k_j})-z_C^{k_j}\|
    =
    \|T(z^{k_j})-R_X(z_C^{k_j})\|
    =
    \|T(z^{k_j})-R_Y(z_C^{k_j})\|.
\end{equation}
By \eqref{tzkct0}, we have
\begin{equation}\label{limtzkjzc}
    \lim_{j \rightarrow +\infty} T(z^{k_j})-z_C^{k_j} =  0,
\end{equation}
which, combined with \eqref{circproper}, gives
\begin{equation}\label{t2gto}
  \lim_{j \rightarrow +\infty}  T(z^{k_j})-R_X(z_C^{k_j}) = 0,
    \qquad
    \lim_{j \rightarrow +\infty} T(z^{k_j})-R_Y(z_C^{k_j}) = 0.
\end{equation}
Using \eqref{zkjqzb}, \eqref{limtzkjzc}, and \eqref{t2gto}, and 
the continuity of the reflections, we deduce
\[
    \bar z=R_X(\bar z)
    \;\;\mbox{ and }\;\;
    \bar z=R_Y(\bar z).
\]
Now $R_X(\bar z)=\bar z$ (resp. $R_Y(\bar z)=\bar z$) is equivalent to 
$\bar z \in X$ (resp. $\bar z \in Y)$.
We have therefore shown that $\bar z \in X \cap Y$ and consequently
every cluster point of \((z_k)\) belongs to \(X \cap Y\).
We now prove that the whole sequence $z_k$ converges to \(\bar z\).

Since \(\bar z\in X \cap Y\), 
applying \eqref{inefejer} with $s=\bar z$ gives
\(
    \|z^{k+1}-\bar z\|
    \le
    \|z^k-\bar z\|, \forall k\ge 0.
\)
Let $\alpha_k:=\|z^k-\bar z\|$. This sequence 
\((\alpha_k)\) is nonnegative nonincreasing. On the other hand, by \eqref{convzbar}, we have
\begin{equation}\label{convalp0}
    \lim_{j \rightarrow +\infty} \|z^{k_j}-\bar z\| = \lim_{j \rightarrow +\infty} \alpha_{k_j}=
    0.
\end{equation}
We now show that this implies \(\lim_{j \rightarrow +\infty} \alpha_{j}=
    0\). Let \(\varepsilon>0\). By \eqref{convalp0}, there exists \(j_0\) such that
$
    \alpha_{k_{j_0}}<\varepsilon$.
Now take any \(k \ge k_{j_0}\). Since \((\alpha_k)\) is nonincreasing, we have
$\alpha_k\le \alpha_{k_{j_0}}<\varepsilon$.
Hence
$\|z^k-\bar z\|<\varepsilon$,
for all $k\ge k_{j_0}$.
It follows that $\lim_{k \rightarrow +\infty} z^k=\bar z \in X \cap Y$.

Therefore, the relaxed cCRM sequence $(z^k)$ converges to a point in \(X\cap Y\).
\end{proof}


\subsection{Relaxed Successive cCRM}


We extend the relaxed iteration to nonempty closed convex sets $C_1,\ldots,C_m\subseteq\mathbb{R}^n$ with $C:=\bigcap_{i=1}^m C_i\neq\varnothing$. For each pair, $T_{C_i,C_j}$ denotes the operator defined in Section~\ref{sec_prelim}. At iteration $k$, two distinct indices $\ell(k),r(k)\in\{1,\ldots,m\}$ are selected according to one of the control rules given in Section~\ref{sec_prelim}, and we define
\begin{equation}\label{tkiter}
T_k:=T_{C_{r(k)},C_{\ell(k)}}.
\end{equation}

The successive centralized circumcentered-reflection method, denoted by ScCRM, takes the form
\begin{equation}\label{sccrm}
z^{k+1}=T_k(z^k).
\end{equation}

Applying the same relaxation to each pairwise step gives the relaxed successive iteration, denoted by RScCRM,
\begin{equation}\label{rsccrm}
z^{k+1}
=
z^k+(1+\varepsilon_k)\bigl(T_k(z^k)-z^k\bigr),
\end{equation}
using the notation $\lambda_k=1+\varepsilon_k$, we obtain, equivalently, $z^{k+1}=(1-\lambda_k)z^k+\lambda_kT_k(z^k)$. Figure~\ref{fig:compar_normal_under_over} illustrates the role
of the relaxation parameter $\lambda_k$. Values
$0<\lambda_k<1$ produce an under-relaxed step, $\lambda_k=1$
recovers the usual cCRM step for the selected pair, and
$\lambda_k>1$ produces an over-relaxed step. For the convergence analysis, we assume the hypothesis from Theorem~\ref{theoremconvrel}.

\begin{figure}[t]
    \centering
\includegraphics[width=0.99\linewidth]{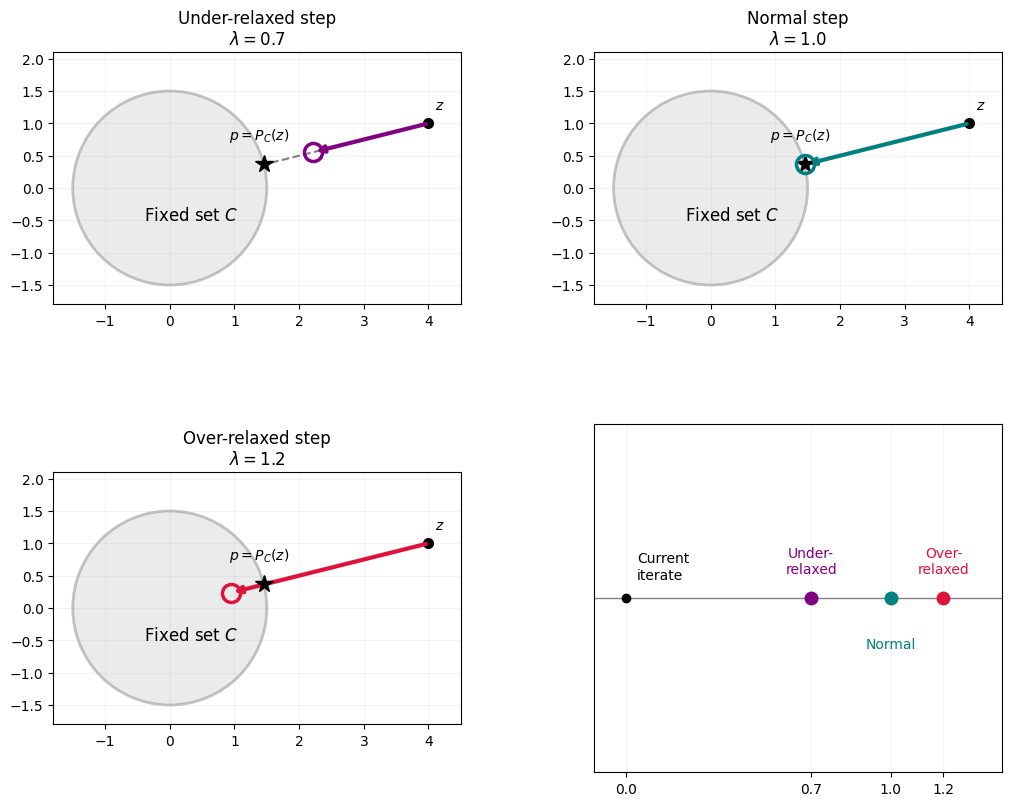}
    \caption{Comparison of the standard, over-relaxed, and under-relaxed cCRM steps.}
    \label{fig:compar_normal_under_over}
\end{figure}

Since $C\subseteq C_{\ell(k)}\cap C_{r(k)}$, Lemma~\ref{lemmarecccrm} can be applied at every iteration with $T_k = T_{C_{r(k)},C_{l(k)}}$ and any $s \in C$. The following projection-residual estimate connects this descent with the first selected constraint.

\begin{lemma}\label{lemccrm2}  For every $s \in A \cap B$ we have
\begin{equation}\label{eqzpab}
\|z-P_B(z)\|^2 \leq \|z-s\|^2 - \|T_{A,B}(z)-s\|^2.
\end{equation}
\end{lemma}
\begin{proof} Take $s \in A \cap B.$
The projection $P_B(z)$
of $z$ onto convex set $B$ satisfies
\begin{equation}\label{eqineq1}
\langle z- P_B(z) , s-P_B(z) \rangle \leq 0. 
\end{equation}
We then have
\begin{eqnarray}
\|z-s\|^2 & = & \|z-P_B(z)\|^2 + \|P_B(z)-s\|^2 + 2 \langle z- P_B(z) , P_B(z)-s \rangle \nonumber \\
&  \stackrel{}{\geq } & \|z-P_B(z)\|^2 + \|P_B(z)-s\|^2. \label{proftcircoper}
\end{eqnarray}
Applying \eqref{proftcircoper}
with $s$ which belongs to $A$, 
$z$ replaced by $P_B(z)$, and 
$B$ by $A$, we obtain
\begin{eqnarray}
\|P_A(P_B(z))-s\|^2 & \leq  & \|P_B(z)-s\|^2 - \|P_A(P_B(z))-P_B(z)\|^2 \leq \|P_B(z)-s\|^2.
\ \label{proftcircopera}
\end{eqnarray}
Next, by nonexpansivity of the
projection and using the fact that $s \in B$, we have
\begin{eqnarray}\label{pbnexp}
\|P_B(P_A(P_B(z)))-s\|=
\|P_B(P_A(P_B(z)))-P_B(s)\|
\leq \|P_A(P_B(z))-s\|.
\end{eqnarray}
Let 
$
c=\frac{1}{2}\Big( P_A(P_B(z))+P_B(P_A(P_B(z)))\Big).
$
By convexity of $\|\cdot\|^2$,
\begin{eqnarray}
\|c-s\|^2 & = & \left \| \frac{1}{2}\Big[ P_A(P_B(z))-s \Big] + \frac{1}{2}\Big[ P_B(P_A(P_B(z)))-s \Big]  \right\|^2  \nonumber \\
& \leq & \frac{1}{2}
\|P_A(P_B(z))-s \|^2 + \frac{1}{2} \| P_B(P_A(P_B(z)))-s\|^2  \nonumber \\
& \stackrel{\eqref{pbnexp}}{\leq} & \|P_A(P_B(z))-s\|^2  \label{circineq1}
\end{eqnarray}
By Lemma 2.5 in \cite{Behling2024Centralization}, we also have
\begin{eqnarray}
\|T_{A,B}(z)-s\|^2 = 
\|\mathcal{C}(c)-s\|^2  & \leq & \|c-s\|^2 \nonumber \\
& \stackrel{\eqref{circineq1}}{\leq} & \|P_A(P_B(z))-s\|^2 \nonumber \\
& \stackrel{\eqref{proftcircopera}}{\leq} & \|P_B(z)-s\|^2 \nonumber \\
& \stackrel{\eqref{proftcircoper}}{\leq} & \|z-s\|^2 - \|z-P_B(z)\|^2 \nonumber \\
\end{eqnarray}
which achieves the proof of the lemma.
\end{proof}

Together with the relaxed descent inequality, Lemma~\ref{lemccrm2} yields feasibility of cluster points under the three control rules considered below.

\begin{theorem}[Convergence of RScCRM]\label{teo_conv_rsccrm}
Let $C_1,\ldots,C_m\subseteq \mathbb{R}^n$ be nonempty closed convex sets and assume that
\(
C:=\bigcap_{i=1}^m C_i\neq\varnothing.
\)
Let $(z^k)$ be generated by
\(
z^{k+1}=z^k+(1+\varepsilon_k)\bigl(T_k(z^k)-z^k\bigr),
\)
for $k \geq 0$
where $T_k:=T_{C_{r(k)},C_{\ell(k)}}$.
Assume that there are
$\eta_1, \eta_2>0$ such that  for all $k \geq k_0$, we have
\begin{equation}\label{eta1eta2}
-1+\eta_1 \leq \varepsilon_k\leq \frac{1}{8}-\eta_2.
\end{equation}
Then the following holds:
\begin{itemize}
\item[A)] If the control sequence is almost cyclic, then $(z^k)$ converges to some point
$z^\star\in C$;
\item[B)] If the most violated constraint (distance function) rule is applied, then $(z^k)$ converges to some point
$z^\star\in C$;
\item[C)] if the most violated constraint (function value version) rule is applied, then $(z^k)$ converges to some point
$z^\star\in C$.
\end{itemize}

\end{theorem}

\begin{proof}
Set
\(
y^k:=T_k(z^k).
\)
The relaxed iteration is
\begin{equation}
z^{k+1}=z^k+\lambda_k(T_k(z^k)-z^k) \mbox{ with }\lambda_k=1+\varepsilon_k.  
\end{equation}
Let $s\in C$. Since
$C\subseteq C_{\ell(k)}\cap C_{r(k)}$,
Lemma \ref{lemmarecccrm} gives
\begin{equation}\label{decsccrm}
\|T_k(z^k)-s\|^2
\leq
\|z^k-s\|^2
-
\frac18\|T_k(z^k)-z^k\|^2.
\end{equation}
Using the identity
$
\|(1-\lambda_k)a+\lambda_k b\|^2
=
(1-\lambda_k)\|a\|^2+
\lambda_k \|b\|^2-
\lambda_k (1-\lambda_k )\|a-b\|^2,
$
with
$
a=z^k-s,
b=T_k(z^k)-s,
$
we obtain
\begin{equation}\label{decnewsccrm}
\|z^{k+1}-s\|^2
=
(1-\lambda_k)\|z^k-s\|^2
+
\lambda_k\|T_k(z^k)-s\|^2
-
\lambda_k(1-\lambda_k)\|T_k(z^k)-z^k\|^2.
\end{equation}
Plugging \eqref{decsccrm}
into \eqref{decnewsccrm} gives for all $k \geq k_0$:
\begin{eqnarray}
\|z^{k+1}-s\|^2
&\leq &
(1-\lambda_k)\|z^k-s\|^2
+
\lambda_k\left(
\|z^k-s\|^2-\frac18\|T_k(z^k)-z^k\|^2
\right) \nonumber \\
&& 
-
\lambda_k(1-\lambda_k)\|T_k(z^k)-z^k\|^2\nonumber\\
& = &
\|z^k-s\|^2
-
\lambda_k\left(\frac18+1-\lambda_k\right)
\|T_k(z^k)-z^k\|^2 \nonumber\\
& \leq &\|z^k-s\|^2
- \eta_1 \eta_2 \|T_k(z^k)-z^k\|^2 \label{firstysuccrc}
\end{eqnarray}
where the last inequality comes
from the relation
$
\lambda_k\left(\frac18+1-\lambda_k\right) \geq \eta_1 \eta_2>0
$
which is an immediate consequence of \eqref{eta1eta2}.
In particular,
$
\|z^{k+1}-s\|\leq \|z^k-s\|$
so $(z^k)_{k \geq k_0}$ is Fejer monotone with respect to $C$.

Summing \eqref{firstysuccrc} from $k=k_0$ to $N>k_0$ gives
\[
\eta_1 \eta_2 \sum_{k=k_0}^N \|T_k(z^k)-z^k\|^2
\leq
\|z^{k_0}-s\|^2-
\|z^{N+1}-s\|^2
\leq
\|z^{k_0}-s\|^2.
\]
Letting $N\to\infty$, we obtain
$
\sum_{k=0}^{\infty}\|T_k(z^k)-z^k\|^2<+\infty$
which implies
\begin{equation}\label{limTkinf1}
\lim_{k \rightarrow +\infty}
\|T_k(z^k)-z^k\| = 0.
\end{equation}
Since
$z^{k+1}-z^k=\lambda_k(T_k(z^k)-z^k)$ and since the sequence $(\lambda_k)$ is uniformly bounded, it follows that
\begin{equation}\label{limTkinf}
\lim_{k \rightarrow +\infty} \|z^{k+1}-z^k\| = 0.
\end{equation}

Since $(z^k)$ is Fejer monotone with respect to the nonempty closed convex set $C$, it is bounded. Therefore it has cluster points. It remains to show that every cluster point belongs to $C$. Once this is established, the standard Fejer monotonicity argument implies that the whole sequence converges to a point of $C$, as was done in the proof of Theorem \ref{theoremconvrel}.

Let $\bar z$ be a cluster point of $(z^k)$. Then there is a subsequence $(z^{k_j})$ such that
\begin{equation}\label{limkj}
\lim_{j \rightarrow + \infty}  z^{k_j} =  \bar z.
\end{equation}
We consider the three cases
A), B), and C) of the theorem. In each case, we want to show that $\bar z \in C$.\\

\par A) The control sequence is almost cyclic. Let $q\in\{1,\ldots,m\}$ be arbitrary. We want to show that
$\bar z \in C_q$.
Since the control is almost cyclic, there exists an integer $Q\geq m$ such that every set is selected at least once in every block of $Q$ consecutive iterations. Hence, for each $j$, there exists an index $h_j$ satisfying
$
k_j\leq h_j\leq k_j+Q$
and
$
\ell(h_j)=q$.

By \eqref{limTkinf}, for every
$\delta>0$ there is
$K$ such that for every 
$k \geq K$ we have
$\|z^{k+1}-z^k\| \leq \delta/Q $.
Take $j$ large enough such that
$h_j \geq k_j \geq K$. Then 
$$
\|z^{h_j}-z^{k_j}\| \leq 
\sum_{p=k_j}^{h_j-1} \|z^{p+1}-z^p\| \leq (h_j-k_j)(\delta/Q) \leq \delta.
$$
This shows 
\begin{equation}
\lim_{j \rightarrow +\infty} \|z^{h_j}-z^{k_j}\|= 0.
\end{equation}
which implies
\begin{equation}\label{convzhjzbar}
\lim_{j \rightarrow +\infty} z^{h_j} =  \bar z.
\end{equation}
Applying Lemma \ref{lemccrm2} to $A=C_{r(k)}$,
$B=C_{\ell(k)}$, $s \in C \subset C_{r(k)} \cap C_{\ell(k)}$,
$T_{A,B}=T_{C_{r(k)},C_{\ell(k)}}=T_k$, we obtain
\begin{equation}\label{cruczkdzclk}
\mbox{dist}(z^k,C_{\ell(k)})^2=\|P_{C_{\ell(k)}}(z^k)-z^k\|^2
\leq
\|z^k-s\|^2-
\|T_k(z^k)-s\|^2.
\end{equation}
Applying \eqref{cruczkdzclk} to 
$k=h_j$ and using the relation $\ell(h_j)=q$, we have
\begin{equation}\label{cruczkdzclk2}
\mbox{dist}(z^{h_j},C_{q})^2
\leq
\|z^{h_j}-s\|^2-
\|T_{h_j}(z^{h_j})-s\|^2.
\end{equation}
Next, the right-hand side in the inequality above
can be written
$$
\|z^{h_j}-s\|^2-
\|T_{h_j}(z^{h_j})-s\|^2 =
\left \langle z^{h_j}-T_{h_j}(z^{h_j}),z^{h_j}+T_{h_j}(z^{h_j})-2s\right \rangle 
$$
which gives, using Cauchy-Schwarz and the triangle inequality
\begin{equation}\label{ineqcsti}
\left \lvert \|z^{h_j}-s\|^2-
\|T_{h_j}(z^{h_j})-s\|^2 \right \rvert \leq 
\|z^{h_j}-T_{h_j}(z^{h_j})\| (\|z^{h_j}-s\|+\|T_{h_j}(z^{h_j})-s\|).
\end{equation}
By  \eqref{limTkinf1}, we have 
$\lim_{j \rightarrow+\infty} z^{h_j}-T_{h_j}(z^{h_j})=0$.
Since both sequences $(z^k)$ and
$(T_{h_j}(z^{h_j}))$ are bounded, we obtain that the
right-hand side and left-hand side
of inequality \eqref{ineqcsti} converge to 0 when 
$j \rightarrow +\infty$, which, combined with 
\eqref{cruczkdzclk2}, gives
\begin{equation}\label{dzhjcq}
\lim_{j \rightarrow +\infty} d(z^{h_j},C_q) = 0.
\end{equation}
Using \eqref{convzhjzbar} and the fact that
$C_q$ is closed, we deduce from \eqref{dzhjcq} that
$\bar z\in C_q$.
Since $q$  was arbitrary,
we have shown 
$\bar z\in \bigcap_{q=1}^m C_q=C$.
Thus every cluster point belongs to $C$. Fejer monotonicity then implies that the whole sequence converges to some point $z^\star\in C$.\\

\par B) Most violated constraint control in distance form.
By definition of the most violated distance rule, for every $i\in\{1,\ldots,m\}$,
\begin{equation}\label{dkjCi}
d(z^{k_j},C_i)
\leq
 d(z^{k_j},C_{\ell(k_j)})
=
\|z^{k_j}-P_{C_{\ell(k_j)}}(z^{k_j})\|.
\end{equation}
As in the proof of item A), 
using Lemma \ref{lemccrm2}, we have
\begin{equation}\label{itemB}
\|z^{k_j}-P_{C_{\ell(k_j)}}(z^{k_j})\|^2 
\leq \|z^{k_j}-s\|^2 - \|T_{k_j}(z^{k_j})-s\|^2,
\end{equation}
and using the fact that
$
\lim_{j \rightarrow+\infty}
T_{k_j}(z^{k_j})-z^{k_j}=0,
$
we deduce
$
\lim_{j \rightarrow +\infty} \|z^{k_j}-P_{C_{\ell(k_j)}}(z^{k_j})\|=0
$
which, combined with 
\eqref{dkjCi}, gives
$
\lim_{j \rightarrow +\infty}
d(z^{k_j},C_i)=0$
for every $i$. 
Using \eqref{convzhjzbar} and the fact that
$C_i$ is closed, we deduce 
$\bar z\in C_i$ for every $i$
and therefore 
$\bar z\in \bigcap_{i=1}^m C_i=C$.
Again, Fejer monotonicity yields convergence of the whole sequence to a point of $C$.\\

\par C) Most violated constraint control in function-value form.
Let $V$ be a neighborhood of 
$\bar z$. Take  $j$ large enough such
that $z^{k_j} \in V$. 
Since functions $f_i$ are 
convex and finite valued, they are Lipschitz
continuous on $V$.
Let $L$ be a Lipscitiz constant for these
functions on $V$. By the most violated function-value rule, for every $i$,
\begin{equation}\label{fiineq1}
f_i(z^{k_j})\leq f_{\ell(k_j)}(z^{k_j}).
\end{equation}
Moreover,
$
P_{C_{\ell(k_j)}}(z^{k_j})\in C_{\ell(k_j)},
$
which implies
\begin{equation}\label{flkplk}
f_{\ell(k_j)}\left(P_{C_{\ell(k_j)}}(z^{k_j})\right)\leq 0.
\end{equation}
Take $i \in \{1,\ldots,m\}$. We have
\begin{eqnarray}
f_i(z^{k_j})
& \stackrel{\eqref{fiineq1},\eqref{flkplk}}{\leq} &
f_{\ell(k_j)}(z^{k_j})
-
f_{\ell(k_j)}\left(P_{C_{\ell(k_j)}}(z^{k_j})\right) \\
&\leq &
L\left\|z^{k_j}-P_{C_{\ell(k_j)}}(z^{k_j})\right\|.
\end{eqnarray}
As for item B), we can show that 
$\lim_{j \rightarrow +\infty} z^{k_j}-P_{C_{\ell(k_j)}}(z^{k_j})=0$ (see \eqref{itemB}) which yields
\begin{equation}\label{limsup}
\limsup_{j\to\infty} f_i(z^{k_j})\leq 0.
\end{equation}
Combining \eqref{limsup} with \eqref{limkj}
gives 
\(
f_i(\bar z)\leq 0.
\)
Therefore $\bar z\in C_i$. Since $i$ was arbitrary,
$
\bar z\in \bigcap_{i=1}^m C_i=C.$
As before, every cluster point belongs to $C$ 
and again Fejer monotonicity implies convergence of the whole sequence to a point of $C$.
\end{proof}

\subsection{Backtracking variant}

Previous schemes, such as RScCRM, define the relaxation parameter
in advance. The backtracking procedure presented below adapts the
extrapolation based on its effect on infeasibility, aiming
to achieve greater progress than the standard cCRM step.
A merit function quantifies constraint violation, with smaller
values indicating lower infeasibility. Each trial point is
accepted if its merit value does not exceed that of the
unrelaxed cCRM point. Otherwise, the extrapolation is reduced,
reverting to the unrelaxed step if no trial point is accepted.
This strategy allows for larger extrapolations when beneficial
and can improve practical performance, at the cost of additional
evaluations of the merit function. Here, $T_k:=T_{C_{r(k)},C_{\ell(k)}}$ denotes the selected
cCRM operator, and $T_k(z^k)$ denotes the computed cCRM point.

Algorithm~\ref{alg:backtracking} considers trial points of the
form $z_{\mathrm{trial}}
=z^k+(1+\varepsilon)\bigl(T_k(z^k)-z^k\bigr)$,
reducing $\varepsilon$ by a factor $\beta\in(0,1)$ after each
unsuccessful trial. Possible merit functions include the
maximum set distance $\max_i d(z,C_i)$ or, the nonnegative maximum violation
$\max\{0,\max_i f_i(z)\}$. The preceding convergence analysis
applies provided that the effective relaxation parameters
satisfy the same eventual uniform bounds, with $\varepsilon=0$
corresponding to the fallback step. Finally, Theorem~\ref{conv_back_variant} below establishes the convergence
of RScCRM with backtracking and follows directly from the
arguments used in the proof of Theorem~\ref{teo_conv_rsccrm}.

\begin{algorithm}[H]
\caption{Backtracking procedure}
\label{alg:backtracking}
\begin{algorithmic}[1]

\Require Closed convex sets $C_1,\ldots,C_m\subset\mathbb{R}^n$;
current iterate $z^k$;
stopping tolerance $\tau>0$;
initial extrapolation parameter $\varepsilon_{\mathrm{initial}}>0$;
reduction factor $\beta\in(0,1)$;
maximum number of backtracking reductions $J_{\max}\in\mathbb{N}$

\State Define a merit function measuring infeasibility
(or a violation function), for example
$\operatorname{merit}(z)
=
\max_{i=1,\ldots,m} d(z,C_i)$,
where $d(z,C_i)$ denotes the distance from $z$ to $C_i$.

\State \textbf{Step 1. Compute the usual centralized circumcenter
reflection method cCRM step.} $T_k(z^k) = \operatorname{cCRM}(z^k)$

\State \textbf{Step 2. Compute the cCRM direction.}
$d_k = T_k(z^k)-z^k$

\State \textbf{Step 3. Compute the merit of the usual cCRM point.} $M_{\mathrm{base}} = \operatorname{merit}(T_k(z^k))$

\State \textbf{Step 4. Initialize the backtracking procedure.}
$\varepsilon = \varepsilon_{\mathrm{initial}};
\mathrm{accepted} = \mathrm{false}$

\State \textbf{Step 5. Backtracking loop.}

\For{$j=0,1,\ldots,J_{\max}$}

    \State Define
    $z_{\mathrm{trial}}
    =
    z^k + (1+\varepsilon)d_k$

    \State Compute
    $M_{\mathrm{trial}}
    =
    \operatorname{merit}(z_{\mathrm{trial}})$

    \If{$M_{\mathrm{trial}}< M_{\mathrm{base}}$}

        \State Accept the extrapolated point:
        $z^{k+1}=z_{\mathrm{trial}};
        \mathrm{accepted}=\mathrm{true}$

        \State Exit the backtracking loop.

    \Else

        \State Reduce the extrapolation parameter:
        $\varepsilon \leftarrow \beta\varepsilon$

    \EndIf

\EndFor

\State \textbf{Step 6. Safeguard/Fallback to the usual cCRM step.}

\If{no extrapolated point was accepted after $J_{\max}$ reductions}

    \State $z^{k+1}=T_k(z^k)$

\EndIf

\State \textbf{Step 7. Stopping criterion.}

\State Check the original stopping criterion of the main algorithm.

\If{it is satisfied}

    \State Stop and \Return $z^{k+1}$.

\Else

    \State Continue to the next iteration.

\EndIf

\end{algorithmic}
\end{algorithm}

\begin{theorem}[Convergence of RScCRM with backtracking]\label{conv_back_variant}
Consider the sequence
$(z^k)$ be generated by
RScCRM with backtracking 
where 
$-1+\eta_1 \leq \varepsilon_{initial} \leq \frac{1}{8}-\eta_2$ for some $\eta_1, \eta_2>0$.
Then the following holds:
\begin{itemize}
\item[A)] If the control sequence is almost cyclic, then $(z^k)$ converges to some point
$z^\star\in C$;
\item[B)] If the most violated constraint (distance function) rule is applied, then $(z^k)$ converges to some point
$z^\star\in C$;
\item[C)] if the most violated constraint (function value version) rule is applied, then $(z^k)$ converges to some point
$z^\star\in C$.
\end{itemize}
\end{theorem}

\begin{proof} 
At iteration $k$ of RScCRM, if
a point is accepted in the
backtracking loop, let $\varepsilon_k$
be the value of 
the corresponding 
parameter $\varepsilon$ (we emphasize that $k$ is not
the number of inner loops
in a backtracking loop and $\varepsilon_k$ is the final value of $\varepsilon$ in the backtracking loop when a point is accepted).
We then define
$\Psi_k$ by
\begin{equation}
\Psi_k= \left\{
\begin{array}{ll}
\varepsilon_k & \mbox{if a relaxed point is accepted at iteration }k,\\
0& \mbox{otherwise.}
\end{array}
\right.
\end{equation}
Then the sequence $(z^k)$
is generated by the sequence
$
z^{k+1}=z^k+(1+\Psi_k)(T_k(z^k)-z^k)
$
which is a sequence generated by RScCRM with relaxed perturbation parameters $\Psi_k$. 
The condition $-1+\eta_2 \leq \varepsilon_{\mathrm{initial}} \leq \frac{1}{8}-\eta_2$ and the fact that $\varepsilon$
decreases in the backtracking inner loops
imply that all parameters
 $\Psi_k$ belong to
$[-1+\eta_1,\frac{1}{8}-\eta_2]$ and the convergence proof of RScCRM without backtracking  applies.
This achieves the proof of the theorem.
\end{proof}

\section{Linear and superlinear convergence rates}\label{sec_lin_sup_con}

Having established global convergence, we now investigate the
convergence rates of RScCRM. We establish linear convergence
under three constraint sselection strategies: most violated
distance control, most violated function value control, and
almost cyclic control. We then examine the additional conditions
under which superlinear convergence holds.

We begin by introducing local error bounds that relate
individual constraint violations to the distance from the
feasible set. These bounds are formulated in terms of distances
or function values, according to the control rule, and are
combined with the relaxed descent estimates to derive the
linear convergence results.

\begin{definition}[EB1: distance error bound - Def. 4.2 from \cite{Behling2024Successive}]
 Let $C_1,\ldots,C_m\subset\mathbb{R}^n$ be nonempty closed convex sets with $C=\bigcap_{i=1}^m C_i\neq\varnothing$. We say that EB1 holds at $z^{\star} \in C$ if there exist $\omega\in(0,1)$ and a neighborhood $V$ of $z^{\star}$ such that
\begin{equation}\label{locerrbdcond}\omega\operatorname{dist}(z,C)\le \max_{1\le i\le m}\operatorname{dist}(z,C_i),\qquad \forall z\in V.
\end{equation}   
\end{definition}

For sets defined by inequalities, we use the corresponding function value form.
\begin{definition}[EB2: function value error bound - Def. 4.3 from \cite{Behling2024Successive}]\label{def.eb2}
If $C_i=\{x \in \mathbb{R}^n:f_i(x)\le0\}$ with convex $f_i:\mathbb{R}^n\to\mathbb{R}$, we say that EB2 holds at $z^{\star} \in C$ if there exist $\omega\in(0,1)$ and a neighborhood $V$ of $z^{\star}$ such that
\[
\omega\operatorname{dist}(z,C)\le \max_{1\le i\le m}f_i(z),\qquad \forall z\in V\setminus C.
\]
\end{definition}

At $z^k$, the most violated function-value rule selects this maximum as $f_{\ell(k)}(z^k)$. The following Fej\'er consequences transfer distance rates to the iterates.

\begin{lemma}[Fej\'er monotonicity and linear convergence - Lem. 4.6 from \cite{Behling2024Successive}]
If $(w^k)$ is Fej\'er monotone with respect to a nonempty closed convex set $M$ and $\operatorname{dist}(w^k,M)$ converges Q-linearly to zero, then $(w^k)$ converges R-linearly to a point of $M$.
\end{lemma}

The superlinear counterpart is analogous.

\begin{lemma}[Fej\'er monotonicity and superlinear convergence - Lem. 4.20 from \cite{Behling2024Successive}]\label{lem_fej_supe_con}
If $(w^k)$ is Fej\'er monotone with respect to a nonempty closed convex set $M$ and $\operatorname{dist}(w^k,M)$ converges superlinearly to zero, then $(w^k)$ converges superlinearly to a point of $M$.
\end{lemma}

The two-set case of EB1 recovers the error-bound condition
used in the analysis of cCRM, which we state explicitly below.

\begin{definition}[Error bound - Def. 4.1 from \cite{Behling2024Successive}]
Let $X, Y \subset \mathbb{R}^n$ be closed convex sets and assume that $X \cap Y \neq \emptyset$. We say that $X$ and $Y$ satisfy a local error bound condition if for some point ${z}^{\star} \in X \cap Y$ there exist a real number $\omega \in(0,1)$ and a neighborhood $V$ of ${z}^{\star}$ such that
$$
\omega \operatorname{dist}(z, X \cap Y) \leq \max \{\operatorname{dist}(z, X), \operatorname{dist}(z, Y)\},
$$
for all $z \in V$.
\end{definition}
To state the rate results precisely, we recall the notions
of Q-linear, Q-superlinear, and R-linear convergence,
together with their associated asymptotic constants.
\begin{definition}[Convergence rate - Def. 4.5 from \cite{Behling2024Successive}]
    Let $\left(w^k\right)_{k \in \mathbb{N}} \subset \mathbb{R}^n$ be a sequence converging to $\bar{w}$. Assume that $w^k \neq \bar{w}$ for all $k \in \mathbb{N}$. Define
$$
q:=\limsup _{k \rightarrow \infty} \frac{\left\|w^{k+1}-\bar{w}\right\|}{\left\|w^k-\bar{w}\right\|}, \quad \text { and } \quad r:=\limsup _{k \rightarrow \infty}\left\|w^k-\bar{w}\right\|^{1 / k} .
$$
Then, the convergence of $\left(w^k\right)_{k \in \mathbb{N}}$ is
\begin{enumerate}
    \item Q-linear if $q \in(0,1)$,
\item $Q$-superlinear if $q=0$,
\item $R$-linear if $r \in(0,1)$.
\end{enumerate}
The values $q, r$ are called asymptotic constants of $\left(w^k\right)_{k \in \mathbb{N}}$.
\end{definition}

\subsection{Linear convergence}


We begin with the most violated distance control rule.
Under this rule, EB1 bounds the distance to the feasible
set in terms of the distance to the selected set.
Combined with the relaxed descent estimate, this bound
yields the following local linear convergence result
for RScCRM.

\begin{theorem}[Local linear convergence under an error bound condition]
Assume that the most violated distance control rule is used and that the following local error bound holds at the limit point $z^\star\in C$: there exist $\omega\in(0,1)$ and a neighborhood $V$ of $z^\star$ such that
\begin{equation}\label{locerrbdcondfirst}
\omega \emph{dist}(z,C)
\leq
\emph{dist}(z,C_i),
 \emph{ for all }z\in V.
\end{equation}
Then, for all sufficiently large $k$,
\begin{equation}\label{dkp1lcrb}
\emph{dist}(z^{k+1},C)
\leq
\sqrt{1-\eta_1 \eta_2 \omega^4}\,\emph{dist}(z^k,C).
\end{equation}
Consequently, $d(z^k,C)$ converges Q-linearly to zero, and the iterates converge R-linearly to a point of $C$.
\end{theorem}

\begin{proof} Recall from \eqref{firstysuccrc} that 
for all $k \geq k_0$, we have
\begin{equation}\label{condlineconv2}
\|z^{k+1}-s\|^2
\leq \|z^k-s\|^2
- \eta_1 \eta_2 \|T_k(z^k)-z^k\|^2.
\end{equation}
Taking $s=P_{C}(z^k)$ in inequality \eqref{condlineconv2} gives
\begin{equation}\label{ratelingood}
\mbox{dist}(z^{k+1},C)^2 \leq 
\mbox{dist}(z^{k},C)^2- \eta_1 \eta_2 \|T_k(z^k)-z^k\|^2.
\end{equation}
For the ScCRM step (see \cite{Behling2024Successive}) under the local error bound, one obtains locally
\begin{equation}\label{scrmlc}
d(T_k(z^k),C)\leq (1-\omega^2)d(z^k,C)
\end{equation}
for all $k$.
Next, by the triangle inequality, we have
\begin{equation}\label{trlinerate}
d(z^k,C)
\leq
\|z^k-T_k(z^k)\|+d(T_k(z^k),C).
\end{equation}
which implies
\begin{equation} \label{finallcvr1}
\|T_k(z^k)-z^k\|
\stackrel{\eqref{trlinerate}}{\geq}
 d(z^k,C)-d(T_k(z^k),C)
\stackrel{\eqref{scrmlc}}{\geq}
\omega^2 d(z^k,C).
\end{equation}
Combining \eqref{ratelingood} and \eqref{finallcvr1}
gives \eqref{dkp1lcrb}.
\end{proof}


We now consider the function value control from Definition~\ref{control_sequences} using sets $C_i=\{x\in\mathbb{R}^n:f_i(x)\le0\}$ where each $f_i$ is convex. The selected indices satisfy
\begin{equation}\label{mvfvcr}
\ell(k)\in\operatorname*{argmax}_{1\le i\le m}f_i(z^k),
\qquad
r(k)\in\operatorname*{argmax}_{1\le i\le m}
f_i\bigl(P_{C_{\ell(k)}}(z^k)\bigr).
\end{equation}

Assume that EB2 in Definition~\ref{def.eb2} holds at the limit point $z^\star\in C$. Using the local boundedness of subgradients, the comparison argument in \cite{Behling2024Successive} yields a constant $\chi>0$ such that, locally,
\begin{equation}\label{ebb22}
d(z,C_{\ell(z)})
\geq \frac{\omega}{\chi} d(z,C),\,\,\ell(z) \in \mbox{argmax}_{1 \leq i \leq m} f_i(z).
\end{equation}
Thus, the function value error bound provides a distance comparison analogous to the one used in the preceding subsection. The following theorem applies this comparison to the relaxed iteration.

\begin{theorem}[Local linear convergence under most violated function value control]
Assume that the most violated function value control rule given by \eqref{mvfvcr} is used. Then, for all $k$ sufficiently large, we have
\begin{equation}\label{dkp1lcr-fvalue}
d(z^{k+1},C)
\leq
\sqrt{1-\eta_1 \eta_2 \left(\frac{\omega}{\chi}\right)^4}\,d(z^k,C).
\end{equation}
\end{theorem}
\begin{proof} Using the analysis from 
\cite{Behling2024Successive} for ScCRM under control rule
\eqref{mvfvcr}, we have
\begin{equation}\label{dkp1lcr1}
d(T_k(z^k),C) \leq \Big(1-\frac{\omega^2}{\chi^2}\Big)d(z^k,C).
\end{equation}
Next, by the triangle inequality
\begin{equation}\label{dkp1lcr2}
d(z^k,C) \leq \|T_k(z^k)-z^k\|+d(T_k(z^k),C).
\end{equation}
Finally, using \eqref{ratelingood}, we conclude that
$$
\mbox{dist}(z^{k+1},C)^2 \leq 
\mbox{dist}(z^{k},C)^2- \eta_1 \eta_2 \|T_k(z^k)-z^k\|^2  \stackrel{\eqref{dkp1lcr1}, \eqref{dkp1lcr2}}{\leq} 
\Big(1-\eta_1 \eta_2 \left(\frac{\omega}{\chi}\right)^4\Big) \mbox{dist}(z^{k},C)^2.
$$
\end{proof}


For almost cyclic control, the analysis tracks progress over
blocks of iterations in which every constraint is selected.
We combine EB1 with the relaxed descent estimate to obtain
a local linear convergence rate through the progress made
over these blocks. In the following theorem, $Q$ denotes
the cycle length in the almost cyclic control rule defined
in Section~2.

\begin{theorem}[Local linear convergence under almost cyclic control]
Let the almost cyclic control rule given by \eqref{mvfvcr} be used.
Assume that sequence $(z^k)$ generated by RScCRM converges to 
$z^{\star} \in C$
and that that the error bound condition \eqref{locerrbdcond} holds.
Define
\begin{equation}\label{defgamma}
\gamma=1-\sqrt{1-\frac{\omega^2}{4}}
\end{equation}
and
\begin{equation}\label{defeta}
\eta = \eta_1 \eta_2 \min\left(\frac{16 \omega^2}{81 Q},\gamma^2\right).
\end{equation}
Then, the sequence $(z^k)$ converges
R-linearly to $z^\star$.
\end{theorem}
\begin{proof} Since $z^k$ converges to some point $z^{\star}$, 
for all $k$ sufficiently large we have that
$z^{k},z^{k+1},\ldots,z^{k+Q}$ belongs to the neighborhood
where \eqref{locerrbdcond} holds. Let $i_k$ such that
\begin{equation}
d(z^k,C_{i_k}) = \max_{1 \leq i \leq m} d(z^k, C_i).
\end{equation}
By the almost cyclic control rule, there is 
$h \in \{k,k+1,\ldots,k+Q-1\}$ such that
$\ell(h)=i_k$.
We consider two cases: 
\begin{equation}\label{caseAlinc}
\mbox{A) } 
\mbox{dist}(z^h,C_{i_k})<(\omega/2)\mbox{dist}(z^k,C)
\end{equation}
and 
\begin{equation}\label{caseBlinc}
\mbox{B) } \mbox{dist}(z^h,C_{i_k})\geq (\omega/2)\mbox{dist}(z^k,C).
\end{equation}

\par A) Case $\mbox{dist}(z^k,C_{i_k})<(\omega/2)\mbox{dist}(z^k,C)$.
Then 
\begin{eqnarray}
\|z^h-z^k\| & \geq & \mbox{dist}(z^k,C_{i_k})-\mbox{dist}(z^h,C_{i_k})\\
& \stackrel{\eqref{locerrbdcondfirst}}{\geq} & \omega \mbox{dist}(z^k,C)-\mbox{dist}(z^h,C_{i_k})\\
& \stackrel{\eqref{caseAlinc}}{\geq} & (\omega/2) \mbox{dist}(z^k,C).
\end{eqnarray}
Next, by Cauchy-Schwarz inequality, we deduce
\begin{eqnarray}
\frac{\omega^2}{4Q}\mbox{dist}(z^k,C)^2 & \leq & \frac{1}{Q}\|z^h-z^k\|^2 = \frac{1}{Q}\|\sum_{t=k}^{h-1} z^{t+1}-z^t\|^2 \\ 
& \leq & \frac{h-k}{Q} \sum_{t=k}^{h-1} \|z^{t+1}-z^t\|^2
\leq \sum_{t=k}^{h-1} \|z^{t+1}-z^t\|^2.
\end{eqnarray}

Summing inequalities
\eqref{ratelingood}
over the cycle
$\{k,k+1,\ldots,k+Q\}$, we obtain
\begin{eqnarray}
\mbox{dist}(z^{k+Q},C)^2 
& \leq & \mbox{dist}(z^{k},C)^2 
-\eta_1 \eta_2  \sum_{t=k}^{h-1} \|T_t(z^t)-z^t\|^2 \\
& = & \mbox{dist}(z^{k},C)^2 
-\eta_1 \eta_2  \sum_{t=k}^{h-1} \frac{1}{(1+\varepsilon_t)^2} \|z^{t+1}-z^t\|^2 \\
& \leq  & \mbox{dist}(z^{k},C)^2 
-\eta_1 \eta_2 \left(\frac{8}{9}\right)^2 \sum_{t=k}^{h-1}  \|z^{t+1}-z^t\|^2 \\
& \stackrel{\eqref{defeta}}{\leq}  & (1-\eta)\mbox{dist}(z^{k},C)^2.
\end{eqnarray}

\par B) Case B) is 
\begin{equation}
\mbox{dist}(z^h,C_{\ell(h)})=\mbox{dist}(z^h,C_{i_k})\geq (\omega/2)\mbox{dist}(z^k,C).
\end{equation}

Using Lemma 3.10 in \cite{Behling2024Successive}, we have 
\begin{equation}\label{lem310scrm}
\|T_{C_{r(h)},C_{\ell(h)}}(z^h)-s\| \leq 
\|P_{ C_{\ell(h)}}(z^h)-s\|\;\;\forall s \in C.
\end{equation}
Using a fundamental
property of the projection onto a convex set, for every 
$s \in C_{\ell(h)}$, we have
\begin{equation}\label{zhsPClh}
\|z^h-s\|^2 \geq \|z^h-P_{C_{\ell(h)}}(z^h)\|^2 
+ \|P_{C_{\ell(h)}}(z^h)-s\|^2.
\end{equation}
Since $C \subset C_{\ell(h)}$,
we can take $s=P_C(z^h)$
in \eqref{zhsPClh} which gives
\begin{equation}\label{zhsPClh1}
\|z^h-P_C(z^h)\|^2 \geq \|z^h-P_{C_{\ell(h)}}(z^h)\|^2 
+ \|P_{C_{\ell(h)}}(z^h)-P_C(z^h)\|^2
\end{equation}
which can be written
\begin{equation}\label{zhsPClh2}
\mbox{dist}(z^h,C)^2 \geq 
\mbox{dist}(z^h,C_{\ell(h)})^2
+ \|P_{C_{\ell(h)}}(z^h)-P_C(z^h)\|^2.
\end{equation}
By definition of the projection, we have
\begin{equation} \label{lemcvl32}
\mbox{dist}(T_{C_{r(h)},C_{\ell(h)}}(z^h),C)
\leq \|T_{C_{r(h)},C_{\ell(h)}}(z^h)-P_C(z^h)\|.
\end{equation}
Writing \eqref{lem310scrm} for
$s=P_C(z^h)$, we obtain
\begin{equation}\label{finallc3}
\|T_{C_{r(h)},C_{\ell(h)}}(z^h)-P_C(z^h)\| \leq 
\|P_{ C_{\ell(h)}}(z^h)-P_C(z^h)\|.
\end{equation}
Combining \eqref{zhsPClh2}, \eqref{lemcvl32}, and \eqref{finallc3}, we have 
\begin{equation}\label{finallc4}
\mbox{dist}(T_{C_{r(h)},C_{\ell(h)}}(z^h),C)^2
\leq \mbox{dist}(z^h,C)^2-\mbox{dist}(z^h,C_{\ell(h)})^2.
\end{equation}
Since sequence $(z^k)$ is Fejer with respect to $C$ and   $h \geq k$, we have 
\begin{equation}\label{zhzk}
\mbox{dist}(z^h,C) \leq \mbox{dist}(z^k,C).
\end{equation}
Next, setting $x=\mbox{dist}(z^h,C)/\mbox{dist}(z^k,C) \in [\omega/2,1]$, we have 
\begin{eqnarray}
\|T_{C_{r(h)},C_{\ell(h)}}(z^h) - z^h\| & \geq  & \mbox{dist}(z^h,C)-\mbox{dist}(T_{C_{r(h)},C_{\ell(h)}}(z^h),C)  \nonumber \\
& \geq  & \mbox{dist}(z^k,C)\left[
\;x-\sqrt{x^2-\frac{\omega^2}{4}}
\right] \nonumber \\
& \geq  & \gamma \mbox{dist}(z^k,C) \label{finalstepBlcv3}
 \end{eqnarray}
where 
$\gamma$ is given by 
\eqref{defgamma}.
Finally, 
\begin{eqnarray}
\mbox{dist}(z^{k+Q},C)^2 & \leq  &\mbox{dist}(z^{h+1},C)^2 \nonumber\\
& \stackrel{\eqref{ratelingood}}{\leq}  &\mbox{dist}(z^{h},C)^2 - \eta_1 \eta_2 \|T_h(z^h)-z^h\|^2 \nonumber\\
& \stackrel{\eqref{zhzk}, \eqref{finalstepBlcv3}}{\leq}  &\mbox{dist}(z^{k},C)^2 - \eta_1 \eta_2 \gamma^2 \mbox{dist}(z^{k},C)^2  \\
& \leq  & (1-\eta) \mbox{dist}(z^{k},C)^2.
 \end{eqnarray}
 Let $\rho=\sqrt{1-\eta} \in (0,1)$.
From A) and B), iterating over blocks,
we have for every $j \geq 0$ that 
$$
\mbox{dist}(z^{k+jQ},C) \leq \rho^j
\mbox{dist}(z^{k},C).
$$
For a given interger $n \geq k$, write
$n=k+jQ + r$ for $0 \leq r < Q$.
Then 
$$
\mbox{dist}(z^{n},C) \leq \mbox{dist}(z^{k+jQ},C) \leq \rho^j
\mbox{dist}(z^{k},C)
$$
and setting $0<\rho^*=\rho^{1/Q}<1$
and using the fact that $j \geq (n-k)/Q-1$, we have 
$$
\mbox{dist}(z^{n},C) \leq \frac{1}{\rho}
(\rho^{*})^{n-k} \mbox{dist}(z^{k},C),
$$
which shows R-linear convergence of sequence
$(z^k)$.
\end{proof}
\subsection{Superlinear convergence}
We now restrict attention to the distance-based and function value most violated constraint rules. The superlinear analysis uses the corresponding results for ScCRM in \cite{Behling2024Successive}. For the distance-based rule, these results require a nonempty interior of the intersection and differentiable boundaries near the limit point. For the function value rule, the corresponding function value error bound is also required.

When $\varepsilon_k\to0$, the relaxed iterate approaches the unrelaxed cCRM point relative to the distance from the feasible set. The following result uses this relation to transfer the superlinear rate to the relaxed method under the stated regularity assumptions.

\begin{theorem}[Superlinear convergence of RScCRM]
Let $(z^k)$ be the sequence
generated by RScCRM.
Assume that 
$\lim_{k \rightarrow +\infty} \varepsilon_k=0$.
Assume that $
\operatorname{int}(C)\neq \varnothing$.
Assume
that
the boundary of each active set $C_i$ is a differentiable manifold in a neighborhood of $z^\star$,
and the corresponding local error-bound assumption required for the chosen control sequence holds.
Then $(z^k)$ converges superlinearly to $z^\star$, namely
\begin{equation}\label{theoremslc}
\lim_{k\to\infty}
\frac{\|z^{k+1}-z^\star\|}{\|z^k-z^\star\|}
=
0,
\end{equation}
where $z^{\star}=\lim_{k \rightarrow +\infty} z^k$.
\end{theorem}

\begin{proof}
Set
$
d_k:=\mbox{dist}(z^k,C)$.
If $d_k=0$ for some sufficiently large $k$, the result is immediate. 
Now assume that
$d_k>0$ for all large $k$.
Define $y^k=T_k(z^k)$. 
Using Lemma~\ref{lemmarecccrm}
for $s=P_C(z^k)$, we have 
\(
\|y^k-s\|^2
\leq
\|z^k-s\|^2
-
\frac18\|y^k-z^k\|^2.
\)
Therefore 
$
\frac18\|y^k-z^k\|^2
\leq
\|z^k-s\|^2
=
d_k^2
$
or equivalently
\begin{equation}\label{slcsqr8}
\|y^k-z^k\|
\leq
2 \sqrt{2}\,d_k .
\end{equation}

The superlinear 
convergence of 
ScCRM gives 
\begin{equation}\label{slcsccrm}
\lim_{k \rightarrow +\infty} \frac{\mbox{dist}(T_k(z^k),C)}
{\mbox{dist}(z^k,C)}= 
\lim_{k \rightarrow +\infty} \frac{\mbox{dist}(y^k,C)}
{d_k}=
0.
\end{equation}

From
$$
z^{k+1}
=
y^k-\varepsilon_k (z^k-y^k),
$$
we obtain, by the triangle inequality, 
\[
d_{k+1}=
\mbox{dist}(z^{k+1},C)
\leq
\mbox{dist}(y^k,C)
+ \varepsilon_k \|z^k-y^k\| 
\stackrel{\eqref{slcsqr8}}{\leq}
\mbox{dist}(y^k,C)+2\sqrt{2} \varepsilon_kd_k.
\]

Dividing by $d_k>0$ gives
\begin{equation}\label{finalslcvr}
0<\frac{d_{k+1}}{d_k}
\leq
\frac{\mbox{dist}(y^k,C)}{d_k}
+ 2\sqrt{2} \varepsilon_k.
\end{equation}
The first term on the right-hand 
of \eqref{finalslcvr}
converges to zero as 
$k \rightarrow +\infty$ due to \eqref{slcsccrm} 
and the second term goes to zero
when $k \rightarrow +\infty$ by assumption. Hence, the scalar distance sequence $d_k=\mbox{dist}(z^k,C)$ converges to zero superlinearly.
Using Lemma~4.20 in \cite{Behling2024Successive}, since sequence $(z^k)$ is Fejér monotone with respect to closed
convex set $C$ and $\mbox{dist}(z^k,C)$ converges to zero superlinearly, we have that $(z^k)$ itself converges
superlinearly to a point of $C$.
Therefore, since $\lim_{k \rightarrow +\infty} z^k=z^{\star} \in C$,
we have \eqref{theoremslc}.
\end{proof}

\section{Perturbed Relaxed Successive cCRM}\label{sec_prsccrm}

The preceding sections established global convergence of RScCRM
and derived linear and superlinear convergence rates under
suitable error-bound, regularity, and relaxation assumptions.
These results describe the asymptotic behavior of the iterates.
We now address finite termination, seeking to guarantee that
an iterate belongs to the original feasible set after finitely
many steps.

To this end, we augment the relaxed successive method with
the PACA-inspired perturbation mechanism of
\cite{Behling2024Finite}, combining inner perturbations with
separating halfspaces. The resulting method, called Perturbed
Relaxed Successive cCRM (PRScCRM), retains the relaxed cCRM
update while replacing the selected original sets with
halfspaces constructed from subgradient linearizations of
the perturbed constraints. These halfspaces vary with the
current iterate and the perturbation parameter. Under the
assumptions stated below, this construction yields finite
termination with respect to the original feasibility problem.


To formulate the perturbed scheme, recall that the feasible sets are represented as
$C_i=\{x\in\mathbb{R}^n:f_i(x)\le0\}$, $i=1,\ldots,m$, where each
$f_i:\mathbb{R}^n\to\mathbb{R}$ is convex, and let
$C=\bigcap_{i=1}^m C_i$. The finite-termination analysis will additionally rely on
strict feasibility, namely, the Slater condition: there exists
$\widehat x\in\mathbb{R}^n$ such that
\[
f_i(\widehat x)<0,
\qquad i=1,\ldots,m.
\]

The perturbation is introduced through a positive sequence
$\{\eta_k\}_{k\ge0}$ satisfying
\begin{equation}\label{assetas}
\eta_k\downarrow0,
\qquad 
\sum_{k=0}^{\infty}\eta_k=+\infty.
\end{equation}
For example, one may take $\eta_k=\nu/(k+1)^r$ with $\nu>0$ and $0<r\le1$.
Associated with this sequence, for each $k$, define the perturbed constraint sets
\begin{equation}
C_i^k:=\{x\in\mathbb{R}^n:f_i(x)+\eta_k\le0\}
\end{equation}
and their intersection
\begin{equation}
C^k:=\bigcap_{i=1}^m C_i^k.
\end{equation}

At the current iterate $z^k$, choose
$g_i^k\in\partial f_i(z^k)$ for each relevant index and define the corresponding
perturbed separating halfspace by
\begin{equation}\label{defhik}
\begin{array}{ll}
H_i^k = \Bigl\{
x\in\mathbb{R}^n:
f_i(z^k)+\langle g_i^k,x-z^k\rangle+\eta_k\le0
\Bigr\} &  \mbox{if }z^k\notin C_i^k,\\
H_i^k=\mathbb{R}^n & \mbox{otherwise.}
\end{array}
\end{equation}

When nonempty, the sets $H_i^k$ admit explicit projection
formulas, making them convenient for the proposed method.

At iteration $k$, choose indices
$\ell(k),r(k)\in\{1,\ldots,m\}$, with the first index selected
according to the most violated function value rule:
\[
    \ell(k)\in
    \operatorname*{argmax}_{1\le i\le m} f_i(z^k).
\]
The second index $r(k)$ may be chosen as another violated constraint, for example, by applying the same most violated rule after projecting onto the first selected halfspace.

Define the selected halfspaces and the corresponding cCRM
operator by
\[
    A_k:=H_{\ell(k)}^k,
    \qquad
    B_k:=H_{r(k)}^k,
    \qquad
    T_k:=T_{B_k,A_k}.
\]
Let $\lambda_k$ be a relaxation parameter satisfying
\(
    \theta_1\leq\lambda_k\leq\frac98-\theta_2
\)
for fixed constants $\theta_1,\theta_2>0$. The perturbed relaxed iteration is then given by
\[
    z^{k+1}
    =
    z^k+\lambda_k\bigl(T_k(z^k)-z^k\bigr) =
    (1-\lambda_k)z^k+\lambda_kT_k(z^k).
\]

Algorithm~\ref{alg:perturbed-relaxed-sccrm} summarizes this
construction, combining the perturbed halfspaces, centralized
circumcenter step, and relaxation, while testing feasibility
with respect to the original set $C$.

\begin{algorithm}[H]
\caption{Perturbed Relaxed Successive Centralized Circumcenter-Reflection Method (PRScCRM)}
\label{alg:perturbed-relaxed-sccrm}
\begin{algorithmic}[1]

\Require $z^0\in\mathbb{R}^n$;
convex functions $f_1,\ldots,f_m$;
perturbations $\eta_k>0$;
relaxation steps $\lambda_k$

\For{$k=0,1,2,\ldots$}

    \If{$z^k\in C=\bigcap_{i=1}^m C_i$}
        \State Stop and return $z^k$.
    \EndIf

    \State Choose
    $\ell(k)\in
    \displaystyle\operatorname*{argmax}_{1\le i\le m} f_i(z^k)$

    \State Choose a second index
    $r(k)\in\{1,\ldots,m\}$.

    \State Pick subgradients
    $g_{\ell(k)}^k\in\partial f_{\ell(k)}(z^k)$ and
    $g_{r(k)}^k\in\partial f_{r(k)}(z^k)$.

    \State Set $A_k := H_{\ell(k)}^k$ and $B_k := H_{r(k)}^k$,
with $H_i^k$ defined by \eqref{defhik}.

    \State Compute the centralized point
    $u^k =
    \dfrac12\Bigl(
    P_{B_k}P_{A_k}(z^k)
    +P_{A_k}P_{B_k}P_{A_k}(z^k)
    \Bigr)$

    \State Compute the reflection
    $v^k=R_{A_k}(u^k)=2P_{A_k}(u^k)-u^k$

    \State Compute the reflection
    $w^k=R_{B_k}(u^k)=2P_{B_k}(u^k)-u^k$

    \State Compute the circumcenter
    $c^k=\mathcal{C}(u^k,v^k,w^k)$

    \State Relax the circumcenter step:
    $z^{k+1}
    =
    z^k+\lambda_k(c^k-z^k)$

\EndFor

\end{algorithmic}
\end{algorithm}

The following lemma establishes the inclusion
$C_i^k\subseteq H_i^k$ for every $i$ and $k$.
\begin{lemma} \label{cikconthik}
For every $k$, we have 
\begin{equation}\label{eqlcikconthik}
C_i^k\subset H_i^k.
\end{equation}
\end{lemma}
\begin{proof}

By convexity of \(f_i\), using the subgradient inequality, 
\[
f_i(z^k)+\langle g_i^k,x-z^k\rangle
\le f_i(x),
\]
which implies that for every \(x\in C_i^k\), we have
\[
f_i(z^k)+\langle g_i^k,x-z^k\rangle+\eta_k
\le f_i(x)+\eta_k
\le0.
\]
This immediately implies \eqref{eqlcikconthik}.
\end{proof}

Consequently, $C^k\subseteq A_k\cap B_k$.
The following result shows that, under Slater's condition,
$C^k$ is nonempty for all sufficiently large $k$.
This guarantees that $A_k\cap B_k\neq\varnothing$ for all
such iterations, ensuring that the corresponding cCRM
steps are well-defined.

\begin{lemma}\label{Cknempty}
There is $k_0$ such that for all
$k \geq k_0$ the set \(C^k\) is nonempty.    
\end{lemma}
\begin{proof}
By Slater assumption, there exists \(\widehat x\) such that
$
f_i(\widehat x)<0, i=1,\ldots,m.
$
Let
\[
\delta:=\min_{1\le i\le m}\{-f_i(\widehat x)\}>0.
\]
Since \(\eta_k\downarrow0\), there exists \(k_0\) such that
$
0<\eta_k<\delta, \forall k\ge k_0.
$

Therefore, for all $k \geq k_0$ we have
$
f_i(\widehat x)+\eta_k<0, i=1,\ldots,m.
$

Consequently, $\hat x \in C_i^k$ for all $k \geq k_0$ and $i=1,\ldots,m,$ or equivalently
$
\widehat x\in C^k, k\ge k_0.
$
\end{proof}

Figure~\ref{fig:sccrm-oprsccrm-geometry} compares the first two
ScCRM and OPRScCRM (Over-Relaxed PRScCRM) iterations for three
ellipsoids in $\mathbb{R}^2$. ScCRM uses the original sets,
whereas OPRScCRM forms perturbed halfspaces and extrapolates
beyond the corresponding circumcenter. For fixed current
point, indices, and perturbation, PRScCRM, OPRScCRM, and
UPRScCRM (Under-Relaxed PRScCRM) share this circumcenter and
differ only through relaxation, with $\lambda_k=1$,
$\lambda_k>1$, and $0<\lambda_k<1$, respectively.

\begin{figure}[t]
    \centering
    \includegraphics[width=1.00\linewidth]{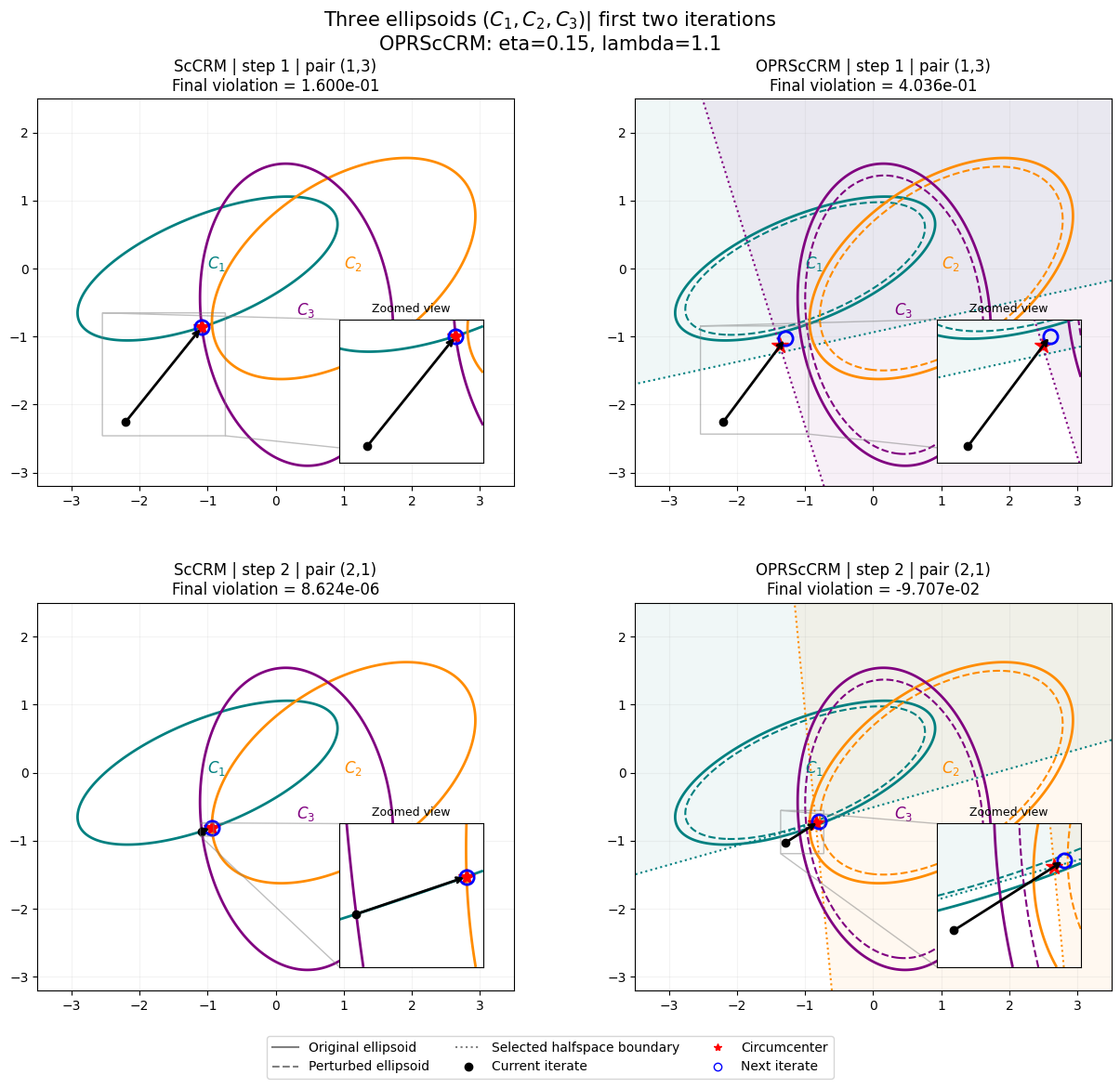}
    \caption{First two iterations of ScCRM (left) and OPRScCRM (right)
for a three-ellipsoid feasibility problem.}
    \label{fig:sccrm-oprsccrm-geometry}
\end{figure}

\subsection{Finite convergence}

This section presents the finite-termination analysis of
PRScCRM. We first establish a descent estimate for the
perturbed relaxed iteration and then use it, together with
the properties of the perturbation sequence, to prove
finite convergence under the stated assumptions.

Since $C^k\subseteq A_k\cap B_k$, Lemma~\ref{lemmarecccrm} applies to the selected halfspaces. Assume
\[
\theta_1\le\lambda_k\le\frac98-\theta_2,
\qquad \theta_1,\theta_2>0.
\]
The same relaxation identity as in Section~\ref{sec_sccrm} yields the following estimate.

\begin{lemma} For every $k \geq 1$,
for every $s \in C^k$, we have
\begin{equation}\label{zkmorefejer}
\|z^{k+1}-s\|^2 \le \|z^k-s\|^2 -\theta_1 \theta_2 \|T_k(z^k)-z^k\|^2.
\end{equation}
\end{lemma}
\begin{proof} 
Take $s \in C^k$. Setting
$
d^k:=T_k(z^k)-z^k$, we can write
\begin{equation}\label{redefzkp1}
z^{k+1}=z^k+\lambda_k d^k.
\end{equation}
Since
\(
C^k\subset A_k\cap B_k,
\) using Lemma~\ref{lemmarecccrm}
we have
\[
\|T_k(z^k)-s\|^2
\le
\|z^k-s\|^2
-
\frac18\|d^k\|^2.
\]

Using the identity
\[
\|(1-\lambda)a+\lambda b\|^2
=
(1-\lambda)\|a\|^2
+
\lambda\|b\|^2
-
\lambda(1-\lambda)\|a-b\|^2,
\]
with
\(
a=z^k-s, 
b=T_k(z^k)-s,
\)
and $\lambda=\lambda_k$,
we obtain
\begin{eqnarray}
\|z^{k+1}-s\|^2
&= &
(1-\lambda_k)\|z^k-s\|^2
+
\lambda_k\|T_k(z^k)-s\|^2
-
\lambda_k(1-\lambda_k)\|d^k\|^2 \nonumber \\
& \leq & \|z^k-s\|^2
-
\lambda_k\left(\frac98-\lambda_k\right)\|d^k\|^2 \nonumber \\ 
& \leq & \|z^k-s\|^2
-
\theta_1 \theta_2 \|T_k(z^k)-z^k\|^2
\end{eqnarray}
where for the first inequality
we have used Lemma \ref{lemmarecccrm}.
\end{proof}

Building on the preceding results, we now state the finite
convergence result for PRScCRM.

\begin{theorem}[Finite convergence]
Assume the following:
\begin{enumerate}
    \item \(C_i=\{x \in \mathbb{R}^n :f_i(x)\le0\}\), where each \(f_i:\mathbb{R}^n\to\mathbb{R}\) is convex.
    \item The Slater condition holds: there exists \(\widehat x\in\mathbb{R}^n\) such that
    \(
    f_i(\widehat x)<0, i=1,\ldots,m.
    \)
    \item The perturbation sequence satisfies
    \(
    \eta_k\downarrow0, \sum_{k=0}^{\infty}\eta_k=+\infty.
    \)
    \item The relaxation parameters satisfy
    \(
    \theta_1 \leq \lambda_k\le \frac98 -\theta_2
    \) for some $\theta_1, \theta_2>0$.
    \item The most violated functional value control is used.
        \item On bounded subsets of \(\mathbb{R}^n\), the selected subgradients are uniformly bounded:
    there exists \(\sigma>0\) such that
    \(
    \|g_i^k\|\le\sigma
    \)
    whenever \(z^k\) remains in the bounded region generated by the algorithm.
    \item A standard error-bound condition holds locally for the family of perturbed separating sets, so that the distance to the perturbed feasible set decreases at least linearly along nonterminating subsequences.

\end{enumerate}
Then the perturbed relaxed successive centralized circumcenter-reflection method terminates finitely. That is, there exists \(K<\infty\) such that
\[
z^K\in C=\bigcap_{i=1}^m C_i.
\]
\end{theorem}

\begin{proof}
We argue by contradiction. Assume that the method does not terminate finitely. Then
\[
z^k\notin C,
\qquad k=0,1,2,\ldots.
\]

\medskip

\textbf{Step 1: Slater implies nonemptiness of the perturbed feasible sets.} By Lemma \ref{Cknempty}, there is $k_0$ such that
for every $k \geq k_0$, we have 
$C^k \neq \emptyset$.\\

\textbf{Step 2: Boundedness.} For \(k\ge k_0\), taking \(s=\widehat x\in C^k\) in 
\eqref{zkmorefejer} gives
\[
\|z^{k+1}-\widehat x\|^2
\le
\|z^k-\widehat x\|^2.
\]
Thus the sequence \(\{z^k\}\) is bounded. Moreover, from
\eqref{zkmorefejer}, we easily deduce 
\[
\sum_{k=k_0}^{\infty}
\|T_k(z^k)-z^k\|^2
<+\infty,
\]
and consequently
\(
\lim_{k \rightarrow +\infty}\|T_k(z^k)-z^k\| = 0.
\)

\medskip

\textbf{Step 3: The perturbation is controlled by the distance to the perturbed feasible set.} Since the algorithm does not terminate, for every \(k\) there exists at least one violated constraint. By the most violated rule,
\[
f_{\ell(k)}(z^k)
=
\max_{1\le i\le m} f_i(z^k)
>0.
\]
By Step 1, for every 
$k \geq k_0$, set
$C^k$
is convex, closed, and nonempty, and we can define
$p^k=P_{C^k}(z^k)$.
Then \(p^k\in C^k\) and therefore
\begin{equation}\label{dist1}
f_{\ell(k)}(p^k)+\eta_k\le0.
\end{equation}
By convexity of \(f_{\ell(k)}\),
\begin{equation}\label{dist2}
f_{\ell(k)}(p^k)
\ge
f_{\ell(k)}(z^k)
+
\langle g_{\ell(k)}^k,p^k-z^k\rangle.
\end{equation}

Combining the two previous inequalities \eqref{dist1} and \eqref{dist2} gives

\[
f_{\ell(k)}(z^k)
+
\langle g_{\ell(k)}^k,p^k-z^k\rangle
+
\eta_k
\le0.
\]
Since
$
f_{\ell(k)}(z^k)>0$,
we obtain
\(
\eta_k
\le
-\langle g_{\ell(k)}^k,p^k-z^k\rangle.
\)
Hence
\[
\eta_k
\le
\left|
\langle g_{\ell(k)}^k,z^k-p^k\rangle
\right|
\le
\|g_{\ell(k)}^k\|\,\|z^k-p^k\|.
\]
Using the boundedness of the subgradients, we have
\(
\|g_{\ell(k)}^k\|\le\sigma,
\)
and 
we obtain
\begin{equation}\label{etakbdzkck}
\eta_k
\le
\sigma\,\mbox{dist}(z^k,C^k).
\end{equation}

\textbf{Step 4: Linear decrease implies summability of \(\eta_k\).} By the relaxed Fejér estimate and the standard error-bound condition, there exist constants $M>0$
and
$\rho\in(0,1)$,
such that, for all \(k\)  sufficiently large, we have
\begin{equation}\label{distckmprho}
\mbox{dist}(z^k,C^k)
\le
M\rho^k.
\end{equation}
Combining 
\eqref{etakbdzkck}
and \eqref{distckmprho}  gives
\(
\eta_k
\le
\sigma M\rho^k.
\)
Therefore,
\[
\sum_{k=0}^{\infty}\eta_k
\leq \sum_{k=0}^{\infty} \sigma M\rho^k < 
+\infty.
\]
This contradicts the assumption
\(
\sum_{k=0}^{\infty}\eta_k=+\infty.
\) Thus the assumption of no finite termination is false. Consequently, there exists \(K<\infty\) such that
\(
z^K\in C.
\)
This achieves the proof.
\end{proof}

Finite termination therefore follows from the combination of nonsummable perturbations, the most violated constraint rule, bounded subgradients, and uniform geometric decay of $\operatorname{dist}(z^k,C^k)$ along a nonterminating run.

\section{Numerical experiments}\label{sec_num_exp}

This section presents numerical experiments on feasibility
problems involving ellipsoids, polyhedra, and Euclidean balls.
We compare execution times and iteration counts to assess
the practical performance of the methods, with particular
emphasis on the effects of relaxation and perturbation.
The methods considered are the cyclic projection method
(Cyc. Proj.), its most-violated-constraint variant
(MV Cyc. Proj.), Cimmino's simultaneous projection method
(Cimmino), and the successive centralized circumcenter-reflection
method (ScCRM). The labels OScCRM and UScCRM denote the over-
and under-relaxed variants of ScCRM, respectively; BOScCRM
and BUScCRM denote their backtracking variants. We also
consider the perturbed approximate circumcenter algorithm
(PACA) and the perturbed successive centralized
circumcenter-reflection method without step relaxation
(PRScCRM), together with its over- and under-relaxed
variants, OPRScCRM and UPRScCRM.

Each method was executed once per problem instance.
Times are reported in seconds, with iteration counts
in parentheses. Boldface blue color in the tables indicates the
smallest execution time among methods satisfying the
stopping criterion for each instance. The labels MI and TL indicate
termination upon reaching the iteration and time limits,
respectively. All methods use a feasibility tolerance as a stopping criterion. However, in every test, the finitely terminating methods returned points satisfying all original constraints.

All experiments were conducted using an Apple M1 processor.
The algorithms were implemented in Julia, with Gurobi 12.0.3
used when numerical optimization was required to compute
projections. The experimental scripts are available at
\url{https://github.com/AdrianaWH/relaxed-ccrm-experiments}.

\subsection{Two ellipsoids with varying angles between the major axes}

We first consider feasibility problems defined by the intersection
of two ellipsoids with fixed shape and size. Their centers and
orientations are chosen symmetrically so that both boundaries
pass through a common point $p$ with a prescribed internal angle
$\alpha$. This angle is the opening between the boundary tangent
rays pointing toward the intersection in the $(x_1,x_2)$-plane.
It is prescribed by choosing outward normals at $p$ that form
the supplementary angle, then positioning and orienting the
ellipsoids accordingly. A point along the internal bisector is
verified to lie strictly inside both sets, ensuring Slater's
condition. The initial point is placed in a fixed direction
relative to the external bisector, with its distance from $p$
increased if necessary to remain outside both sets.
Figure~\ref{fig:ellipsoid-angles} illustrates this geometry.


Varying the intersection angle allows us to assess how
intersection geometry affects the performance of the methods.
Small angles can slow down cyclic projections, which may
repeatedly zigzag between the sets while making little
progress toward their intersection. This comparison therefore
examines whether the proposed methods encounter similar
difficulties or are less sensitive to small intersection angles. The initial point is placed $10$ units
from the common boundary point at $70^\circ$ to the external
bisector, increasing its distance if needed to remain outside both sets.

As Table~\ref{tab:ellipsoid-angles} shows, Cyc. Proj. and
MV Cyc. Proj. reach the iteration limit at $10^\circ$, whereas
the perturbed methods terminate in at most nine reported iterations
in every case. Although not always the fastest, the proposed variants
achieve execution times close to the best results in several cases.
Their low iteration counts, particularly at small angles, suggest
that they avoid the prolonged zigzag behavior associated with
cyclic projections.

\begin{figure}[t]
\centering
\includegraphics[width=1.0\linewidth]{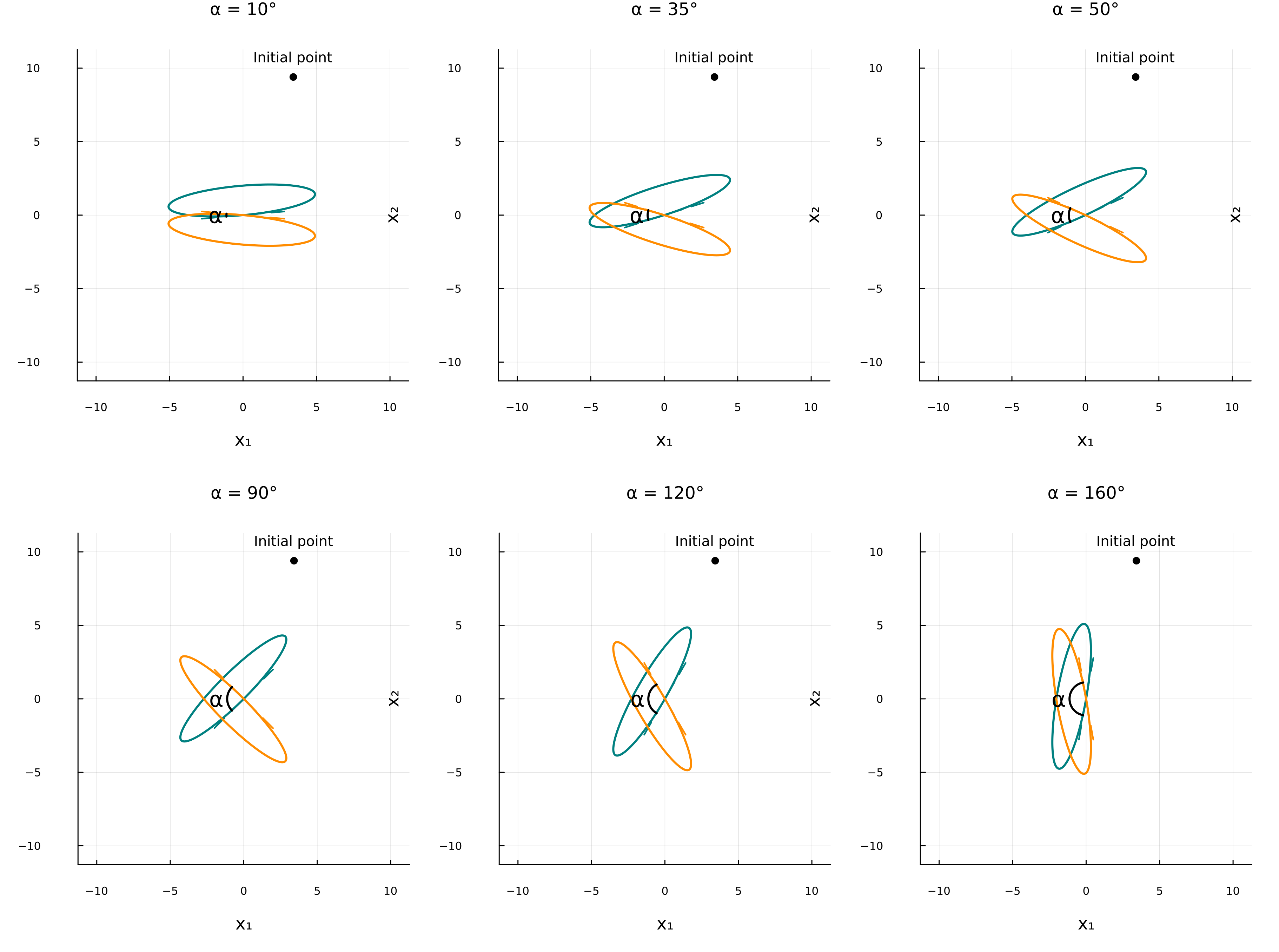}
\caption{Two-ellipsoid configurations and initial points, with tangents defining the internal angle $\alpha$. Dashed lines indicate tangents at the common boundary point.}
\label{fig:ellipsoid-angles}
\end{figure}

\begin{table}[t]
\caption{Two-Ellipsoid experiments with controlled intersection angles with $10^{-6}$ tolerance. MI indicates reaching $50\,000$ iterations.}
\label{tab:ellipsoid-angles}
\centering
\scriptsize
\setlength{\tabcolsep}{3pt}
\begin{tabular}{@{}lcccccc@{}}
\toprule
Method & $10^\circ$ & $35^\circ$ & $50^\circ$ & $90^\circ$ & $120^\circ$ & $160^\circ$ \\
\midrule
Cyc. Proj.
& 15.3880 (MI) & 0.0067 (23) & 0.0036 (12)
& 0.0005 (1) & 0.0022 (6) & 0.0031 (9) \\
MV Cyc. Proj.
& 16.9817 (MI) & 0.0060 (21) & 0.0003 (1)
& 0.0005 (1) & 0.0020 (6) & 0.0034 (10) \\
Cimmino
& 17.2490 (MI) & 0.0240 (85) & 0.0118 (41)
& 0.0028 (16) & 0.0073 (26) & 0.0105 (37) \\
ScCRM
& 0.0035 (4) & 0.0026 (4) & 0.0008 (2)
& 0.0009 (2) & 0.0028 (4) & 0.0037 (4) \\
OScCRM
& 0.1150 (6) & 0.0027 (4) & 0.0020 (4)
& 0.0021 (4) & 0.0029 (4) & 0.0023 (4) \\
BOScCRM
& 0.0049 (3) & 0.0018 (3) & 0.0007 (2)
& 0.0009 (2) & 0.0027 (3) & 0.0022 (3) \\
UScCRM
& 0.0874 (10) & 0.0073 (10) & 0.0062 (11)
& 0.0067 (10) & 0.0085 (9) & 0.0075 (9) \\
BUScCRM
& 0.0029 (4) & 0.0032 (4) & 0.0007 (2)
& 0.0008 (2) & 0.0023 (3) & 0.0039 (4) \\
PACA
& 0.1144 (45) & \textcolor{blue}{\textbf{0.0001}} (8) & \textcolor{blue}{\textbf{0.0001}} (8)
& \textcolor{blue}{\textbf{0.0001}} (6) & \textcolor{blue}{\textbf{0.0001}} (5) & \textcolor{blue}{\textbf{0.0001}} (6) \\
PRScCRM
& 0.2130 (6) & 0.0016 (6) & 0.0017 (6)
& 0.0013 (5) & \textcolor{blue}{\textbf{0.0001}} (5) & 0.0004 (5) \\
OPRScCRM
& \textcolor{blue}{\textbf{0.0009}} (5) & 0.0020 (4) & 0.0007 (3)
& 0.0004 (2) & \textcolor{blue}{\textbf{0.0001}} (2) & 0.0015 (7) \\
UPRScCRM
& 0.0022 (9) & 0.0025 (9) & 0.0031 (9)
& 0.0027 (9) & 0.0004 (9) & 0.0010 (8) \\
\bottomrule
\end{tabular}
\end{table}

\subsection{Ellipsoidal feasibility problems}

Next, we consider intersections of $m$ randomly generated ellipsoids in $\mathbb{R}^n$, represented as $C_i=\{x:(x-y_i)^\top Q_i(x-y_i)\le k_i^2\}$ and generated to ensure a common feasible intersection. All methods use the same instances and terminate execution when $(x-y_i)^\top Q_i(x-y_i)\le(k_i+\varepsilon)^2$ for all ellipsoids. We compare the relaxed and backtracking variants against ScCRM, and the perturbed variants against PACA. Table~\ref{tab:ellipsoids-original-large-radius} shows that OPRScCRM yields the shortest execution time in 18 of the 24 instances, while PACA is the fastest in the remaining six.

\begin{sidewaystable}
\caption{Random ellipsoid feasibility experiments with tolerance
$10^{-7}$, at most $500\,000$ iterations (MI), and a
$100$-second time limit (TL).}
\label{tab:ellipsoids-original-large-radius}
\centering
\scriptsize
\setlength{\tabcolsep}{2pt}
\begin{tabular*}{\textheight}{@{\extracolsep{\fill}}rr*{9}{c}@{}}
\toprule
$m$ & $n$ & ScCRM & OScCRM & BOScCRM & UScCRM & BUScCRM
& PACA & PRScCRM & OPRScCRM & UPRScCRM \\
\midrule
5 & 20 & 0.0022 (2) & 0.0021 (2) & 0.0024 (2) & 0.0125 (7) & 0.0024 (2) & \textcolor{blue}{\textbf{0.0001}} (5) & 0.0010 (4) & 0.0004 (2) & 0.0402 (11) \\
10 & 20 & 0.0040 (2) & 0.0028 (2) & 0.0033 (2) & 0.0257 (8) & 0.0022 (2) & \textcolor{blue}{\textbf{0.0002}} (6) & 0.0017 (5) & 0.0007 (2) & 0.0063 (11) \\
20 & 20 & 0.0052 (3) & 0.0060 (3) & 0.0026 (2) & 0.0176 (8) & 0.0045 (3) & 0.1427 (8) & 0.0017 (6) & \textcolor{blue}{\textbf{0.0007}} (3) & 0.0034 (12) \\
50 & 20 & 0.0044 (3) & 0.0042 (3) & 0.0036 (2) & 0.0157 (8) & 0.0051 (3) & \textcolor{blue}{\textbf{0.0009}} (8) & 0.0022 (6) & 0.0012 (3) & 0.0060 (12) \\
\midrule
5 & 50 & 0.0136 (3) & 0.0064 (2) & 0.0067 (2) & 0.0592 (8) & 0.0125 (3) & \textcolor{blue}{\textbf{0.0003}} (6) & 0.0017 (6) & 0.0007 (3) & 0.0042 (13) \\
10 & 50 & 0.0151 (3) & 0.0073 (2) & 0.0078 (2) & 0.0447 (8) & 0.0154 (3) & 0.0084 (7) & 0.0022 (6) & \textcolor{blue}{\textbf{0.0010}} (3) & 0.0049 (13) \\
20 & 50 & 0.0217 (4) & 0.0071 (2) & 0.0079 (2) & 0.0519 (8) & 0.0293 (4) & \textcolor{blue}{\textbf{0.0009}} (7) & 0.0027 (6) & 0.0013 (3) & 0.0068 (13) \\
50 & 50 & 0.0144 (3) & 0.0134 (3) & 0.0093 (2) & 0.0479 (7) & 0.0240 (3) & 0.0039 (8) & 0.0051 (6) & \textcolor{blue}{\textbf{0.0025}} (3) & 0.0135 (13) \\
\midrule
5 & 100 & 0.0511 (3) & 0.0364 (2) & 0.0376 (2) & 0.1568 (7) & 0.0515 (3) & \textcolor{blue}{\textbf{0.0006}} (6) & 0.0022 (5) & 0.0008 (2) & 0.0070 (13) \\
10 & 100 & 0.0518 (3) & 0.0358 (2) & 0.0354 (2) & 0.1743 (8) & 0.0503 (3) & 0.0024 (6) & 0.0052 (6) & \textcolor{blue}{\textbf{0.0017}} (2) & 0.0115 (13) \\
20 & 100 & 0.0606 (3) & 0.0345 (2) & 0.0350 (2) & 0.1839 (8) & 0.0674 (3) & 0.0040 (7) & 0.0102 (6) & \textcolor{blue}{\textbf{0.0015}} (2) & 0.0212 (13) \\
50 & 100 & 0.0857 (4) & 0.0412 (2) & 0.0425 (2) & 0.1870 (8) & 0.1181 (4) & 0.0132 (8) & 0.0171 (6) & \textcolor{blue}{\textbf{0.0043}} (2) & 0.0445 (13) \\
\midrule
5 & 200 & 0.2524 (3) & 0.1296 (2) & 0.1226 (2) & 0.9210 (8) & 0.2274 (3) & 0.0027 (6) & 0.0068 (5) & \textcolor{blue}{\textbf{0.0014}} (2) & 0.0231 (13) \\
10 & 200 & 0.2632 (3) & 0.1345 (2) & 0.1375 (2) & 0.8336 (7) & 0.3203 (3) & 0.0068 (6) & 0.0114 (5) & \textcolor{blue}{\textbf{0.0027}} (2) & 0.0366 (13) \\
20 & 200 & 0.3624 (3) & 0.1645 (2) & 0.2045 (2) & 0.9646 (7) & 0.3723 (3) & 0.0172 (6) & 0.0262 (6) & \textcolor{blue}{\textbf{0.0055}} (2) & 0.0623 (13) \\
50 & 200 & 0.8347 (4) & 0.1712 (2) & 0.1919 (2) & 1.0314 (7) & 0.6881 (4) & 0.0612 (7) & 0.0716 (6) & \textcolor{blue}{\textbf{0.0138}} (2) & 0.1884 (14) \\
\midrule
5 & 500 & 2.3127 (4) & 0.8433 (2) & 0.6809 (2) & 4.8135 (8) & 1.9524 (4) & 0.2311 (7) & 0.0534 (7) & \textcolor{blue}{\textbf{0.0300}} (4) & 0.1231 (15) \\
10 & 500 & 2.2345 (4) & 0.6913 (2) & 0.7998 (2) & 5.0743 (8) & 2.3104 (4) & 0.0834 (7) & 0.1290 (7) & \textcolor{blue}{\textbf{0.0445}} (3) & 0.3334 (16) \\
20 & 500 & 3.3427 (5) & 0.7260 (2) & 0.7929 (2) & 5.7506 (8) & 4.9338 (5) & 0.3714 (9) & 0.5106 (8) & \textcolor{blue}{\textbf{0.1543}} (3) & 1.0776 (16) \\
50 & 500 & 4.0905 (5) & 0.9929 (2) & 1.2897 (2) & 7.3037 (8) & 9.8874 (5) & 1.3040 (8) & 1.9666 (8) & \textcolor{blue}{\textbf{0.5478}} (3) & 4.1427 (16) \\
\midrule
5 & 1000 & 7.9809 (3) & 3.8758 (2) & 4.1825 (2) & 20.1921 (6) & 8.4820 (3) & 0.1403 (7) & 0.2021 (6) & \textcolor{blue}{\textbf{0.0768}} (3) & 0.5210 (15) \\
10 & 1000 & 12.8683 (4) & 4.1516 (2) & 4.3978 (2) & 24.7491 (7) & 14.2157 (4) & 0.3477 (7) & 0.5701 (7) & \textcolor{blue}{\textbf{0.2109}} (3) & 1.3920 (16) \\
20 & 1000 & 17.0523 (5) & 4.5442 (2) & 4.6654 (2) & 27.2741 (7) & 28.4846 (5) & 1.3002 (7) & 1.9576 (7) & \textcolor{blue}{\textbf{0.7042}} (3) & 5.3473 (16) \\
50 & 1000 & 21.5625 (5) & 5.5978 (2) & 6.5131 (2) & 37.7668 (8) & 45.3544 (5) & 5.4569 (8) & 6.8389 (7) & \textcolor{blue}{\textbf{2.2445}} (3) & 18.4990 (17) \\
\bottomrule
\end{tabular*}
\end{sidewaystable}

\subsection{Polyhedral feasibility problems}


We consider the feasibility problem $C=\bigcap_{i=1}^m P_i$,
where $P_i=\{x\in\mathbb{R}^n:A_i x\le b_i\}$ has $p=2n$
inequalities. A common point $\widehat{x}$ and the entries of
each $A_i$ are sampled from the standard normal distribution.
Each row of $A_i$ is then normalized to unit Euclidean norm,
preserving its random direction. We set
$b_i=A_i\widehat{x}+\alpha r_i$, where $\alpha>0$ and the
components of $r_i$ are positive uniform random draws.
Because the rows have unit norm, each slack $\alpha(r_i)_j$
equals the distance from $\widehat{x}$ to the corresponding
constraint hyperplane. This provides direct control over these
distances and makes $\widehat{x}$ strictly feasible for every
polyhedron. All methods use the same initial point and stop when
$\max_{i,j}(A_i x-b_i)_j\le\mathrm{tol}$ or when the iteration
or time limit is reached. For this problem, we compare PACA with variants of PRScCRM. Table~\ref{tab:synthetic-polyhedra} shows that the PRScCRM variants become particularly competitive as $m$ increases. Furthermore, it is important to note that no method reaches the tolerance for $(m,n)=(50,500)$ or $(200,100)$ within the established limits.

\begin{table}[t]
\caption{Synthetic polyhedral experiments with tolerance $10^{-6}$.
MI and TL denote limits of $30\,000$ iterations and $600$ seconds,
respectively.}
\label{tab:synthetic-polyhedra}
\centering
\scriptsize
\setlength{\tabcolsep}{3pt}
\begin{tabular}{@{}rrrcccc@{}}
\toprule
$m$ & $n$ & $p$ & PACA & PRScCRM & OPRScCRM & UPRScCRM \\
\midrule
5 & 20 & 40 & 0.1825 (43) & 0.3400 (49) & 0.0084 (51) & \textcolor{blue}{\textbf{0.0079}} (43) \\
5 & 50 & 100 & \textcolor{blue}{\textbf{0.0129}} (52) & 0.1647 (85) & 0.0290 (80) & 0.0305 (100) \\
5 & 100 & 200 & \textcolor{blue}{\textbf{0.0427}} (98) & 0.1383 (184) & 0.1160 (160) & 0.1749 (228) \\
5 & 500 & 1000 & \textcolor{blue}{\textbf{8.0920}} (466) & 23.1876 (1044) & 23.5331 (1089) & 40.6616 (1367) \\
\midrule
10 & 20 & 40 & \textcolor{blue}{\textbf{0.0124}} (99) & 0.0207 (84) & 0.0310 (96) & 0.0158 (50) \\
10 & 50 & 100 & 0.0835 (244) & 0.0798 (143) & 0.0799 (173) & \textcolor{blue}{\textbf{0.0579}} (132) \\
10 & 100 & 200 & \textcolor{blue}{\textbf{0.1267}} (168) & 0.2633 (162) & 0.2274 (173) & 0.2562 (191) \\
10 & 500 & 1000 & \textcolor{blue}{\textbf{8.1499}} (240) & 45.4502 (834) & 47.7163 (792) & 64.4511 (1035) \\
\midrule
20 & 20 & 40 & 0.0943 (424) & \textcolor{blue}{\textbf{0.0674}} (403) & 0.0961 (626) & 0.0790 (535) \\
20 & 50 & 100 & 0.3497 (705) & 0.3302 (441) & 0.3688 (483) & \textcolor{blue}{\textbf{0.3175}} (436) \\
20 & 100 & 200 & 1.2021 (788) & 1.1748 (517) & 1.1605 (486) & \textcolor{blue}{\textbf{1.0854}} (451) \\
20 & 500 & 1000 & \textcolor{blue}{\textbf{107.6366}} (770) & 161.8536 (854) & 179.5122 (823) & 158.6638 (916) \\
\midrule
50 & 20 & 40 & 1.2811 (7595) & 0.8346 (3149) & 0.8844 (3475) & \textcolor{blue}{\textbf{0.7764}} (3362) \\
50 & 50 & 100 & 5.2995 (4768) & 3.2936 (1877) & 3.4435 (1963) & \textcolor{blue}{\textbf{3.0764}} (1808) \\
50 & 100 & 200 & 19.1724 (4493) & \textcolor{blue}{\textbf{16.1875}} (2516) & 17.3843 (2667) & 16.7730 (2533) \\
\midrule
100 & 20 & 40 & 3.3333 (9424) & 3.2488 (5680) & 3.5586 (7906) & \textcolor{blue}{\textbf{2.4885}} (5685) \\
100 & 50 & 100 & 41.7141 (18314) & \textcolor{blue}{\textbf{26.9078}} (7865) & 33.6744 (9483) & 36.1643 (9642) \\
100 & 100 & 200 & 279.9911 (MI) & \textcolor{blue}{\textbf{185.6361}} (13053) & 196.7902 (13633) & 193.3795 (13648) \\
\midrule
200 & 20 & 40 & 18.8263 (27109) & \textcolor{blue}{\textbf{9.1373}} (10500) & 14.6750 (16221) & 13.0382 (14382) \\
200 & 50 & 100 & 152.2522 (MI) & 167.7795 (22806) & \textcolor{blue}{\textbf{161.8577}} (22843) & 178.2050 (25390) \\

\bottomrule
\end{tabular}
\end{table}

\subsection{Feasibility experiments with intersecting euclidean balls}

As a final case study, we consider the feasibility problem of finding a point in the intersection of \(m\) Euclidean balls in \(\mathbb{R}^n\). Their radii \(R_i\) are sampled independently on a logarithmic scale within a prescribed interval, allowing for substantial differences in size. Given a small common inner radius \(\rho>0\), with \(\rho<\min_i R_i\), each center is defined as \(y_i=-(R_i-\rho)u_i\), where \(u_i\) is a unit vector. This ensures that each ball contains \(B(0,\rho)\) and that the origin satisfies Slater's condition. The first direction defined for constructing the balls is random and the second is its opposite; subsequent directions are selected using an angular dispersion procedure. When \(\rho\) is small relative to the radii, the opposite pair (the second to be chosen) produces balls that are nearly externally tangent with a narrow overlap, introducing a controlled geometric challenge while preserving strict feasibility. Instances are rejected if any center lies within the full intersection, and a common starting point is positioned outside all the balls.

Table~\ref{tab:balls-opposite-centers} shows mixed performance:
the proposed variants achieve the shortest execution time in 11 of
the 20 instances, including all cases with $n\ge100$ and $m\le50$. PRSCCRM and UPRSCCRM complete all tested instances without reaching
either limit. In the instances where the proposed variants are fastest,
they also terminate after few iterations.

\begin{table}[t]
\caption{Euclidean-ball experiments with tolerance $10^{-8}$ and initial point at least $100$ from every ball. MI indicates reaching $10\,000$ iterations and TL the $100$-second limit.}
\label{tab:balls-opposite-centers}
\centering
\scriptsize
\setlength{\tabcolsep}{3pt}
\begin{tabular}{@{}rrcccccc@{}}
\toprule
$n$ & $m$ & MV Cyc. Proj. & Cimmino & PACA
& PRScCRM & OPRScCRM & UPRScCRM \\
\midrule
20 & 5 & 0.0471 (2211) & 0.0980 (MI) & 0.0245 (1932) & 0.0031 (12) & \textcolor{blue}{\textbf{0.0017}} (8) & 0.0037 (15) \\
20 & 10 & \textcolor{blue}{\textbf{0.0003}} (16) & 0.0186 (782) & 0.0004 (19) & 0.0085 (34) & 0.0075 (28) & 0.0078 (49) \\
20 & 20 & 0.0144 (285) & 0.3738 (MI) & \textcolor{blue}{\textbf{0.0017}} (25) & 0.0109 (37) & 0.0242 (57) & 0.0095 (42) \\
20 & 50 & 0.0036 (17) & 0.2648 (2691) & \textcolor{blue}{\textbf{0.0025}} (27) & 0.0311 (102) & 0.0167 (51) & 0.0322 (95) \\
20 & 100 & \textcolor{blue}{\textbf{0.0005}} (3) & 0.6804 (3661) & 0.0045 (27) & 0.0405 (94) & 0.0548 (119) & 0.0404 (98) \\
\midrule
50 & 5 & 0.0511 (973) & 0.4075 (9310) & 0.0217 (509) & 0.0028 (11) & \textcolor{blue}{\textbf{0.0011}} (7) & 0.0048 (16) \\
50 & 10 & 0.0672 (506) & 0.9545 (9721) & 0.0147 (153) & 0.0062 (14) & \textcolor{blue}{\textbf{0.0033}} (10) & 0.0073 (19) \\
50 & 20 & 0.0094 (31) & 0.2801 (1156) & \textcolor{blue}{\textbf{0.0028}} (18) & 0.0217 (50) & 0.0196 (38) & 0.0318 (66) \\
50 & 50 & 0.0186 (29) & 1.5546 (2937) & \textcolor{blue}{\textbf{0.0113}} (21) & 0.0692 (74) & 0.0573 (57) & 0.0670 (68) \\
50 & 100 & 0.0294 (22) & 6.1512 (6035) & \textcolor{blue}{\textbf{0.0214}} (21) & 0.2252 (142) & 0.2299 (138) & 0.2248 (136) \\
\midrule
100 & 5 & 0.1014 (464) & 0.7488 (4455) & 0.0198 (122) & 0.0058 (10) & \textcolor{blue}{\textbf{0.0042}} (10) & 0.0081 (15) \\
100 & 10 & 0.1957 (395) & 2.8735 (7593) & 0.0255 (65) & \textcolor{blue}{\textbf{0.0127}} (14) & 0.0493 (50) & 0.0154 (18) \\
100 & 20 & 0.5898 (580) & 7.4400 (MI) & 0.1420 (185) & \textcolor{blue}{\textbf{0.0249}} (16) & 0.3031 (181) & 0.0364 (24) \\
100 & 50 & 2.3074 (892) & 18.7518 (MI) & 0.7319 (369) & \textcolor{blue}{\textbf{0.1234}} (41) & 1.3440 (408) & 0.1289 (42) \\
100 & 100 & 0.0988 (17) & 22.8345 (5062) & \textcolor{blue}{\textbf{0.0868}} (20) & 0.7278 (112) & 0.6425 (98) & 0.7070 (106) \\
\midrule
200 & 5 & 8.9315 (MI) & 7.3932 (MI) & 6.5045 (MI) & 0.0776 (11) & 16.0285 (MI) & \textcolor{blue}{\textbf{0.0211}} (15) \\
200 & 10 & 2.4068 (1166) & 15.4468 (MI) & 1.0082 (662) & \textcolor{blue}{\textbf{0.0402}} (12) & 1.7437 (613) & 0.0642 (21) \\
200 & 20 & 9.3071 (2158) & 32.9392 (MI) & 5.2186 (1552) & \textcolor{blue}{\textbf{0.0977}} (18) & 0.1063 (19) & 0.1247 (24) \\
200 & 50 & 13.2352 (827) & [TL] (9048) & 2.8540 (262) & 0.8491 (46) & 4.7569 (280) & \textcolor{blue}{\textbf{0.7406}} (47) \\
200 & 100 & 0.4341 (14) & 74.9550 (3520) & \textcolor{blue}{\textbf{0.2993}} (16) & 2.1814 (75) & 2.4866 (87) & 2.2054 (78) \\
\bottomrule
\end{tabular}
\end{table}
\section{Conclusions}\label{sec_conclusions}

We introduced relaxed cCRM and relaxed Successive cCRM variants for
convex feasibility problems. Geometrically, relaxation adjusts
the distance traveled along the cCRM direction, allowing the
iterate to stop before or move beyond the circumcenter.
In particular, controlled extrapolation can reach points
closer to the feasible region than the unrelaxed step,
potentially improving practical performance. The relaxed
two-set step preserves a Fej\'er-type descent estimate on
an explicit parameter interval, yielding global convergence
for under-relaxed and controlled over-relaxed iterations.
The successive analysis extends these conclusions to almost
cyclic and most violated distance/function-value controls,
with linear and superlinear rates under the corresponding
local error-bound and regularity assumptions.

To promote finite termination, we combined relaxation with
PACA-inspired inner perturbations and separating halfspaces.
Under the stated Slater, subgradient, relaxation, and
error-bound assumptions, the finite-termination mechanism
follows from the incompatibility between nonsummable
perturbations and geometric decay of the distance to the
perturbed feasible sets. 

The numerical experiments complement the analysis: the
perturbed variants generally reach the prescribed tolerance
with low iteration counts and competitive execution times.

Future work includes deriving verifiable uniform error
bounds for the varying perturbed sets, refining adaptive
relaxation strategies, extending the numerical study to
additional convex feasibility classes, and investigating
momentum-enhanced cCRM and successive variants.

\bibliographystyle{plain}
\bibliography{bibliography}

\end{document}